\documentclass[11pt]{article}
\usepackage{graphicx}
\usepackage{import}
\usepackage{caption}
\usepackage{calc}

\newtoks\by
\newtoks\paper
\newtoks\book
\newtoks\jour
\newtoks\yr
\newtoks\pages
\newtoks\vol
\newtoks\publ

\def\name[#1, #2]{#1 #2}
\def\ota{{\hbox{\bf ???}}}
\def\cLear{\by=\ota\paper=\ota\book=\ota\jour=\ota\yr=\ota
\pages=\ota\vol=\ota\publ=\ota}
\def\endpaper{\the\by, \textit{\the\paper},
{\the\jour} \textbf{\the\vol} (\the\yr), \the\pages.\cLear}
\def\endbook{\the\by, \textit{\the\book},
\the\publ, \the\yr.\cLear}
\def\endpap{\the\by, \textit{\the\paper}, \the\jour.\cLear}
\def\endproc{\the\by, \textit{\the\paper}, \the\book, \the\publ,
\the\yr, \the\pages.\cLear}

\usepackage{geometry}
\usepackage{graphicx}
\usepackage{amsmath,amssymb,amsthm}
\usepackage{enumerate}
\usepackage{hyperref}
\usepackage[active]{srcltx}
\usepackage[T1]{fontenc}
\usepackage{color}
\newtheorem{theo}{Theorem}
\newtheorem{algo}{Algorithm}
\newtheorem{lem}{Lemma}[section]
\newtheorem{defi}{Definition}[section]
\newtheorem{cor}{Corollary}[section]
\newtheorem{prop}{Proposition}[section]
\newtheorem{rmk}{Remark}[section]
\newcommand{\eps}{\varepsilon}

\newcommand{\R}{\mathbb{R}}
\newcommand{\Z}{\mathbb{Z}}

\numberwithin{equation}{section}
\begin{document}
\title{An algorithm for one-dimensional Generalized Porous Medium Equations: interface tracking and the hole filling problem}
\date{}
\author{
   L\'eonard Monsaingeon
  \footnote{
    CAMGSD, Instituto Superior T\'ecnico, Av. Rovisco Pais 1049-001 Lisboa, Portugal
    \texttt{leonard.monsaingeon@ist.utl.pt} }
}

\maketitle

\begin{abstract}
Based on results of E. DiBenedetto and D. Hoff we propose an explicit finite difference scheme for the one dimensional Generalized Porous Medium Equation $\partial_t u=\partial_{xx}^2 \Phi(u)$. The scheme allows to track the moving free boundaries and captures the hole filling phenomenon when two free boundaries collide. We give an abstract convergence result when the mesh parameter $\Delta x\to 0 $ without any error estimates, and invesigate numerically the convergence rates.
\end{abstract}

\section{Introduction}
We consider the numerical approximation of nonnegative solutions $u(x,t)\geq 0$ to one-dimensional degenerate diffusion equations of the Generalized Porous Medium Equation type
\begin{equation}
\partial_t u=\partial^2_{xx} \Phi(u),\qquad t\geq 0,x\in \R.
\tag{GPME}
\label{eq:GPME_u}
\end{equation}
The nonlinearity $\Phi(s)$ is normalized as $\Phi(0)=0$, is monotone increasing for $s>0$, and satisfies the structural condition
\begin{equation}
1<a\leq \frac{s\Phi'(s)}{\Phi(s)} \leq b
 \tag{$\Gamma_{a,b}$}
\label{eq:strucural_ab} 
\end{equation}
for some constants $a,b$. This roughly means that nonlinearities in the class $\Gamma_{a,b}$ behave in between two pure powers $s^a,s^b$ for $1<a\leq b$, which is a generalization of the celebrated Porous Medium Equation (PME) $\partial_tu=\Delta u^m$ for $m>1$. Moreover, $a>1$ implies that $\Phi(s)/s$ is monotone increasing and $\lim\limits_{s\to 0^+}\frac{\Phi(s)}{s}=\Phi'(0)=0$. Writing $\partial_{xx}^2\Phi(u)=\partial_x(\Phi'(u)\partial_x u)$ the equation clearly degenerates at the levelset $\{u=0\}$, which results in the so-called \emph{finite speed of propagation}: if the initial data $u^0(x)$ is compactly supported then $u(\,.\,,t)$ remains compactly supported for all $t>0$, see \cite{dPV91}. Thus free-boundaries $\Gamma(t)=\partial \operatorname{supp} u(\,.\,,t)$ separate $\{u =0\}$ from $\{u>0\}$. In order to understand their propagation it is more convenient to use the \emph{pressure} variable, defined as
$$
v:=\Psi(u),\qquad \Psi(s):=\int_0^s\frac{\Phi'(z)}{z}dz.
$$
The pressure formally solves
\begin{equation}
v_t=\sigma(v)\partial^2_{xx}v+|\partial_xv|^2 ,
 \label{eq:GPME_v}
\end{equation}
where
$$
\sigma(v)=\Phi'(u)=\Phi'\circ \Psi^{-1}(v).
$$
The structural assumption \eqref{eq:structural_condition_sigma} implies that $(a-1)v\leq \sigma(v)\leq (b-1)v$, and $v,\sigma(v),\Phi'(u),\Phi(u)/u$ are comparable in the sense that the ratio of any two of them is bounded away from zero and from above. As a consequence $u=v=\sigma(v)=0$ at the free-boundaries, and formally discarding the $\sigma(v)\partial_{xx}^2v$ term we see that $\partial_t v=|\partial_x v|^2$ at any free-boundary point. This suggests that the free-boundary curves $\zeta(t)=\partial \operatorname{supp}v(\,.\,,t)$ should propagate with local speed $d\zeta/dt=-\partial_{x}v(\zeta(t),t)$, provided that these quantities make sense. As a consequence the speed of propagation should be bounded as soon as the pressure is Lipschitz in the space variable. Degenerate diffusion equations such as \eqref{eq:GPME_u} have attracted considerable attention in the last decades. We refer the reader to \cite{Va07,DK86,DK07,S83,DB93} and references therein 
for the 
Cauchy 
problem and regularity theory, and to \cite{A70,CVW87,CW90,DR03,dPV91} 
for the theory of free-boundaries.

In order to track the free-boundaries we shall work exclusively in the pressure framework \eqref{eq:GPME_v} rather than with \eqref{eq:GPME_u}, and we restrict in the whole paper to Lipschitz-continuous and compactly supported initial pressure
$$
0\leq v^0(x)\leq M,\qquad \operatorname{Lip}(v^0)\leq \gamma_0.
$$
Because \eqref{eq:GPME_u} and \eqref{eq:GPME_v} satisfy a comparison principle \cite{Va07} we expect that $0\leq v(x,t)\leq M$ for all times and the behaviour of $\Phi(s)$ should therefore be irrelevant for large $r=\Psi(s)\geq M$. As a consequence we relax \eqref{eq:strucural_ab} and only assume throughout the whole paper
$$
\sigma\in \mathcal{C}^1([0,\infty),\R^+)\cap\mathcal{C}^2(\R^+,\R^+),\qquad \sigma(0)=0,\qquad \sigma'>0,
$$
and
\begin{equation}
\forall \,r\in[0,M]:\qquad
0<s_1(M)\leq  \sigma'(r)\leq S_1(M)
\quad\mbox{and}\quad
|\sigma''(r)|\leq S_2(M)
\label{eq:structural_condition_sigma}
\end{equation}
for structural $s_1,S_1,S_2$.
\begin{rmk}
This condition on $\sigma(r)$ can be translated into conditions on the original $\Phi(s)$ nonlinearity through $r=\Psi(s)$, for example $\sigma'(r)=s\Phi''(s)/\Phi'(s)$. In the case of the pure PME nonlinearity $\Phi(s)=s^m$ one can compute explicitly $v=\Psi(u)=mu^{m-1}/(m-1)$ and $\sigma(v)=(m-1)v$, thus $s_1=S_1=(m-1)$ and $S_2=0$ in \eqref{eq:structural_condition_sigma}. As a consequence the above structural assumptions for $\sigma$ can be viewed as some PME-like behaviour condition in bounded intervals.
\end{rmk}
Because of gradient jumps at the free-boundaries no classical solutions can exist if $v^0$ has compact support, and we shall use the following weak formulation:
\begin{defi}
A function $0\leq v\in\mathcal{C}(\R\times [0,T])$ is a weak solution of \eqref{eq:GPME_v} with initial datum $v^0(x)$ if $\partial_x v\in L^2(\R\times(0,T))$ and
\begin{align*}
& \int\limits_{\R}v(x,\tau)\varphi(x,\tau)\mathrm{d}x-\int\limits_{\R}v^0(x)\varphi(x,0)\mathrm{d}x\nonumber\\
&
 \qquad+ \int\limits_{0}^{\tau}\int\limits_{\R}\left\{-v \partial_t\varphi	+ \sigma(v)\partial_x v\partial_x\varphi+\Big(1-\sigma'(v)\Big)|\partial_x v|^2\varphi\right\}\mathrm{d}x\,\mathrm{d}t=0
 \label{eq:weak_formulation_v_IBP0}
 \end{align*}
 for all $0\leq \tau\leq T$ and test functions $\varphi\in \mathcal{C}^{\infty}_c(\R\times[0,T])$.
 \label{defi:weak_sols_v}
\end{defi}
The equivalence between the density $u$ and pressure $v$ formulations with $v=\Psi(u)$ is well known \cite{A69}, and any weak solution $v$ in the sense of Definition~\ref{defi:weak_sols_v} automatically gives a weak solution $u=\Psi^{-1}(v)$ to \eqref{eq:GPME_u} in some sense. As already mentioned we only work in the pressure variable, hence we refrain from giving a precise definition of weak solutions for \eqref{eq:GPME_u} and refer the reader e.g. to \cite{Va07,dPV91}. Note that we impose here continuity at $t=0^+$, so that the initial data are taken in a strong sense.\\

The problem of numerical approximation to \eqref{eq:GPME_v} in dimension one goes back to \cite{GJ71}, where a finite difference approach was first proposed to compute numerical solutions of $\partial_t v=f(x,t,v)\partial_{xx}^2v+|\partial_x v|^2$ but free-boundaries were not accurately tracked. Later in \cite{TM83} a scheme allowing to track the interfaces was implemented for the pure PME nonlinearity $\Phi(s)=s^m$, but the authors were not able to prove convergence of the interface curves. Almost simultaneously, DiBenedetto and Hoff proposed in \cite{DBH84} an explicit finite-difference interface-tracking algorithm for the pure PME nonlinearity, and established rigorous error estimates for the solution and interfaces. In \cite{DBH84,GJ71,TM83} only the case of initial data $v^0$ consisting in a single patch is considered, i-e with when the initial support only has one connected component $\operatorname{supp}v^0=[\zeta_l(0),\zeta_r(0)]$. In this case the free-boundaries can be represented by two continuous 
left/right curves $\zeta_{lr}(t)$ with $\operatorname{supp}v(\,.\,,t)=[\zeta_l(t),\zeta_r(t)]$ for all $t\geq 0$. It is well known \cite[Corollary 1.5]{dPV91} that due to the diffusive nature of the problem $\operatorname{supp}v(\,.\,,t)$ is noncontracting in time, and as a consequence $\zeta_l$ and $\zeta_r$ are monotone nonincreasing and nondecreasing respectively.
In addition to this simple setting we will also consider the so-called \emph{hole-filling problem} when the initial support has two connected components at positive distance from each other, in which case the internal hole eventually fills and the internal interfaces disappear in finite time (see section~\ref{section:two_patches} for a detailed description of the problem). A finite elements method was recently employed in \cite{QZ09} to investigate the hole-filling and related problems, with satisfactory qualitative results but no rigorous convergence result.

Closely following \cite{DBH84}, we propose in this paper an extension of DiBenedetto and Hoff's algorithm to general nonlinearities, allowing to track the interfaces and solve past the hole-filling time.
As in \cite{DBH84} the algorithm reproduces at the discrete level all the properties satisfied by the solutions of \eqref{eq:GPME_v} at the continuous level.
More precisely: initial $\gamma_0$ Lipschitz regularity, nonnegativity, and $L^{\infty}$ bounds are preserved along the time evolution, solutions are $1/2$ H\"older continuous in time, and satisfy a generalized Aronson-B\'enilan estimate $\partial_{xx} v(\,.\,,t)\geq\underline{z}(t)\approx -C(1+1/t)$ in the sense of distributions $\mathcal{D}'(\R)$ for all fixed $t>0$.
For the pure PME nonlinearity $\Phi(s)=s^m$ the latter semi-convexity property was first proved in \cite{AB79} in the optimal form $\partial_{xx} v(x,t)\geq-1/(m+1)t$, and is fundamental for the regularity and propagation theories.
The scheme relies on the following splitting method: inside the support $\{v>0\}=\{\sigma(v)>0\}$ \eqref{eq:GPME_v} is formally parabolic, hence a classical finite difference scheme can be used with an extra $\eps$-viscosity stabilizing term. As already discussed one formally expects the hyperbolic propagation law $d\zeta/dt=-\partial_x v$ at the free-boundaries $x=\zeta(t)$, and thus enforcing the discrete equivalent allows to track the interfaces. Technically speaking this interface condition is in fact applied at the discrete level in some neighborhood of the interface curves. The neighborhood has thickness of the same order $\mathcal{O}(\Delta x)$ as the space mesh $\Delta x$, and can therefore be viewed as a numerical boundary layer.\\

The paper is organized as follows: in Section~\ref{section:one_patch} we describe the scheme for general nonlinearities when the initial data consists in a single patch (i-e has connected initial support). Imposing a suitable stability condition $\Delta t=\mathcal{O}(\Delta x^2)$ on the mesh parameters we establish discrete a priori bounds, including a generalized Aronson-B\'enilan estimate (Lemma~\ref{lem:Aronson_Benilan_estimate}). These a priori estimates then allow us to prove convergence of the approximate solutions and interface curves when $h=(\Delta x,\Delta t)\to 0$. In Section~\ref{section:two_patches} we show that the scheme can be extended to study the hole-filling problem. We construct a numerical approximation to the filling time and show that our scheme really captures the hole-filling phenomenon, in the sense that it allows to keep computing a consistent approximation to the solution past the filling time. In Section~\ref{section:num_exp_comments} we present a numerical experiments and 
investigate the order of convergence.

As already mentioned Section~\ref{section:one_patch} is an adaptation of \cite{DBH84} to general nonlinearities but requires significant technical modifications, in particular for the generalized Aronson-B\'enilan estimate (Lemma~\ref{lem:Aronson_Benilan_estimate}). To the best of our knowledge all the results in Section~\ref{section:two_patches} are new, even for the pure PME nonlinearity.
%
%
%
\section{The scheme for one patch only}
\label{section:one_patch}
%
%
Throughout the whole paper we fix mesh parameters $\Delta x,\Delta t$ and write $\{x_k\}_{k\in\Z}=\{k\,\Delta x\}$, $\{t^n\}_{n\geq 0}=\{n\,\Delta t\}$, as well as $v^n_k\approx v(x_k,t^n)$ and $\zeta_{lr}^n\approx\zeta_{lr}(t^n)$. Given a ``single patch'' compactly supported initial datum $v^0$
$$
0\leq v^0(x)\leq M,\qquad
\operatorname{Lip}(v^0)\leq \gamma_0,
\qquad
\operatorname{supp}v^0=[\zeta_l(0),\zeta_r(0)],
$$
we first initialize
$$
v^0_k:=v^0(x_k)
\qquad \mbox{and}\qquad
\zeta_{l,r}^0:=\zeta_{l,r}(0).
$$
Given an approximate solution $v_k^n$ and interfaces $\zeta_{l,r}^n$ at time $t^n$, we define
$$
K_l(n):=\min\{k\in \Z:\,x_{k-1}\geq \zeta_l^n\},
\qquad
K_r(n):=\max\{k\in \Z:\,x_{k+1}\leq \zeta_r^n\}
$$
and
$$
0\leq s^n_l:=x_{K_l(n)}-\zeta_l^n,\qquad 0\leq s^n_r:=\zeta_r^n-x_{K_r(n)}.
$$
We shall often speak of $x_k\in [x_{K_l(n)},x_{K_r(n)}]$ as the (numerical) support at time $t^n$, while $x_k\in[\zeta_l^n,x_{K_l(n)}]$ and $x_k\in[x_{K_r(n)},\zeta_r^n,]$ will be referred to as the (numerical) left and right boundary layers. Observe that by construction these boundary layers have thickness $\Delta x\leq s_l^n,s_r^n\leq 2\Delta x$, see Figure~\ref{fig:FIG1}. The interfaces at time $t^{n+1}$ are next computed as
\begin{equation}
\frac{\zeta_l^{n+1}-\zeta_l^n}{\Delta t}=-\frac{v_{K_l(n)}^n}{s_l^n},
\qquad\qquad
\frac{\zeta_r^{n+1}-\zeta_r^n}{\Delta t}=-\frac{v_{K_r(n)}^n}{s_r^n},
\label{eq:propagation_interfaces}
\end{equation}
thus reproducing the propagation law $d\zeta/dt=-\partial_x v$ at the free-boundaries. We will prove in Lemma~\ref{lem:Linfty_Lipschitz_estimate} that $v_k^n\geq 0$, and therefore $\zeta_{l}^{n+1}\leq \zeta_l^n$ and $\zeta_{r}^{n+1}\geq \zeta_r^n$. This monotonicity translates the noncontractivity of the support at the discrete level. We also define for later use
$$
\left(s^{n}_l\right)':=x_{K_l(n)}-\zeta_l^{n+1}\geq s^n_l,\qquad \left(s^{n}_r\right)':=\zeta_r^{n+1}-x_{K_r(n)}\geq s_r^n.
$$
Carefully note that $\left(s^{n}_{lr}\right)'\neq s^{n+1}_{lr}$ and that $\zeta^{n}_{lr}$ needs not be integer meshpoints, see Figure~\ref{fig:FIG1}.
\begin{figure}
 \def\svgwidth{.7\textwidth}
 \begin{center}
 \begingroup%
  \makeatletter%
  \providecommand\color[2][]{%
    \errmessage{(Inkscape) Color is used for the text in Inkscape, but the package 'color.sty' is not loaded}%
    \renewcommand\color[2][]{}%
  }%
  \providecommand\transparent[1]{%
    \errmessage{(Inkscape) Transparency is used (non-zero) for the text in Inkscape, but the package 'transparent.sty' is not loaded}%
    \renewcommand\transparent[1]{}%
  }%
  \providecommand\rotatebox[2]{#2}%
  \ifx\svgwidth\undefined%
    \setlength{\unitlength}{841.88974609bp}%
    \ifx\svgscale\undefined%
      \relax%
    \else%
      \setlength{\unitlength}{\unitlength * \real{\svgscale}}%
    \fi%
  \else%
    \setlength{\unitlength}{\svgwidth}%
  \fi%
  \global\let\svgwidth\undefined%
  \global\let\svgscale\undefined%
  \makeatother%
  \begin{picture}(1,0.70707072)%
    \put(0,0){\includegraphics[width=\unitlength]{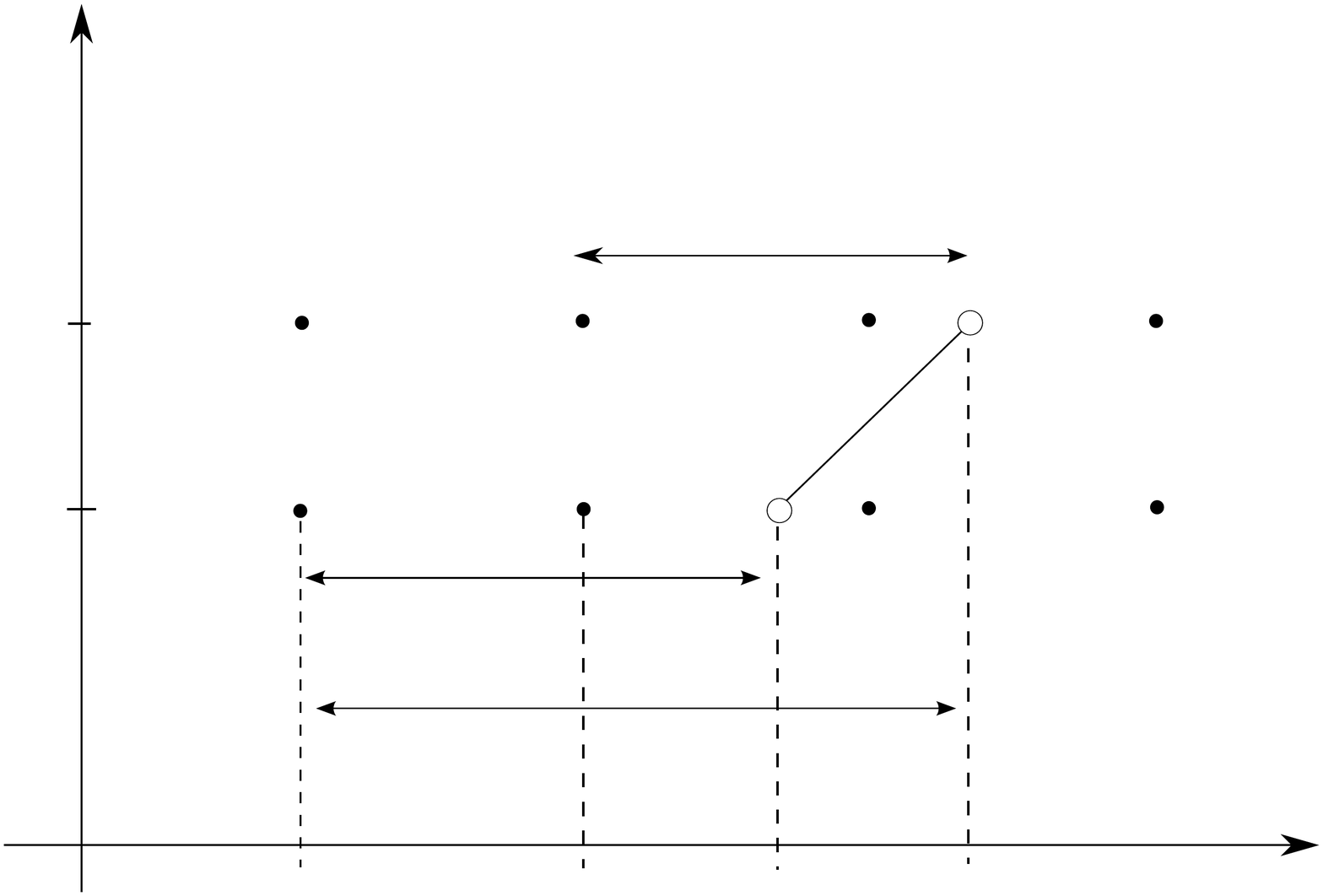}}%
    \put(0.92520172,0.04161024){\color[rgb]{0,0,0}\makebox(0,0)[lb]{\smash{$x$}}}%
    \put(0.20724699,0.04058544){\color[rgb]{0,0,0}\makebox(0,0)[lb]{\smash{$x_{K_r(n)}$}}}%
    \put(0.56787001,0.04191117){\color[rgb]{0,0,0}\makebox(0,0)[lb]{\smash{$\zeta_r^n$}}}%
    \put(0.70581748,0.04160674){\color[rgb]{0,0,0}\makebox(0,0)[lb]{\smash{$\zeta_r^{n+1}$}}}%
    \put(0.01925342,0.32353758){\color[rgb]{0,0,0}\makebox(0,0)[lb]{\smash{$t^n$}}}%
    \put(0.0059475,0.45515401){\color[rgb]{0,0,0}\makebox(0,0)[lb]{\smash{$t^{n+1}$}}}%
    \put(0.01888279,0.66697707){\color[rgb]{0,0,0}\makebox(0,0)[lb]{\smash{$t$}}}%
    \put(0.39631151,0.2460684){\color[rgb]{0,0,0}\makebox(0,0)[lb]{\smash{$s_r^n$}}}%
    \put(0.45548712,0.14654446){\color[rgb]{0,0,0}\makebox(0,0)[lb]{\smash{$\left(s^n_r\right)'$}}}%
    \put(0.54257394,0.53529407){\color[rgb]{0,0,0}\makebox(0,0)[lb]{\smash{$s_r^{n+1}$}}}%
    \put(0.38510369,0.04028106){\color[rgb]{0,0,0}\makebox(0,0)[lb]{\smash{$x_{K_r(n+1)}$}}}%
  \end{picture}%
\endgroup%

 \end{center}
 \caption{right numerical boundary layer}
\label{fig:FIG1}
 \end{figure}
The solution $v^{n+1}_k$ is then updated inside the support by enforcing
\begin{equation}
k\in[K_l(n),K_r(n)]:\qquad
\frac{v^{n+1}_k-v^n_k}{\Delta t}=(\sigma(v^n_k)+\eps)\frac{v_{k-1}^{n}-2v_{k}^{n}+v_{k+1}^{n}}{\Delta x^2}+\left|\frac{v_{k+1}^{n}-v_{k-1}^{n}}{2\,\Delta x}\right|^2,
 \label{eq:scheme}
\end{equation}
where $\eps>0$ is a fixed artificial viscosity parameter to be chosen later. Observe that \eqref{eq:scheme} is not applied across the interfaces but only in the numerical support, where \eqref{eq:GPME_v} is formally in the parabolic regime since $\{v>0\}=\{\sigma(v)>0\}$.

Inside the boundary layers of thickness $\left(s^n\right)'$ the solution is interpolated as
\begin{equation}
 \qquad v_{k}^{n+1}:=\left\{
 \begin{array}{cl}
  v_{K_l(n)}^{n+1}\frac{x_k-\zeta_l^{n+1}}{x_{K_l(n)}-\zeta_l^{n+1}} \qquad & x_k\in[\zeta^{n+1}_l,x_{K_l(n)-1}]\\
  v_{K_r(n)}^{n+1}\frac{\zeta_r^{n+1}-x_k}{\zeta_r^{n+1}-x_{K_r(n)}} & x_k\in[x_{K_r(n)+1},\zeta^{n+1}_r]
 \end{array}
 \right.,
\label{eq:def_linear_interpolation}
\end{equation}
and finally we set
$$
v_k^{n+1}:=0\qquad \mbox{for }x_k\notin[\zeta_l^{n+1},\zeta_r^{n+1}].
$$
The interpolation \eqref{eq:def_linear_interpolation} is consistent with the well known linear behaviour of the pressure variable across the moving free boundaries \cite[Theorem 15.24]{Va07}, see Lemma~\ref{lem:linear_growth_interface} later on.
\begin{rmk}
According to \eqref{eq:def_linear_interpolation} $v_k^n$ is exactly linear in the boundary layers. As a consequence \eqref{eq:propagation_interfaces} also reads $\frac{\zeta_l^{n+1}-\zeta_l^n}{\Delta t}=-\frac{v_{K_l(n)}-v_{K_l(n)-1}}{\Delta x}$ and $\frac{\zeta_r^{n+1}-\zeta_r^n}{\Delta t}=-\frac{v_{K_r(n)+1}-v_{K_r(n)}}{\Delta x}$, again reproducing the propagation law $d\zeta/dt=-\partial_x v$.
\label{rmk:propagation_law_interpolation}
\end{rmk}
Throughout the whole paper and without further mention we impose the following Courant-Fredriech-Lewis stability condition

\begin{equation}
 \begin{array}{c}
    \frac{\Delta t}{\Delta x^2}:=\beta\leq \frac{1}{2\Big( \sigma(M)+\eps\Big)
	+\gamma_0\Delta x\Big(4+3S_1(M)\Big)
	+ \gamma_0^2\Delta x^2 S_2(M)/2}\\
     	\gamma_0\Delta x\Big(27+9s_1(M)+3S_1(M)+\Delta xS_2(M)/4\Big)\leq \eps\leq\mathcal{O}(\Delta x)
  \end{array}
    \label{eq:CFL}
\tag{CFL} 
\end{equation}
%
with $\|v^0\|_{L^{\infty}(\R)}\leq M$, $\operatorname{Lip}(v^0)\leq \gamma_0$, and $s_1(M),S_1(M),S_2(M)\geq 0$ as in \eqref{eq:structural_condition_sigma}. 
%
%
%
\subsection{A priori discrete estimates}
\label{subsection:estimates_one_patch}
Defining the discrete downwind and centered spatial derivatives
$$
w^n_k:=\frac{v^n_k-v^n_{k-1}}{\Delta x},\qquad \overline{w}^n_k:=\frac{v^n_{k+1}-v^n_{k-1}}{2\Delta x},
$$
the first discrete estimate reads
\begin{lem}
Assume that $0\leq v^0_k\leq M$ with $|w^0_k|\leq \gamma_0$. Then for all $k,n$ there holds
$$
0\leq v^n_k\leq M\quad \mbox{and}\quad |w^n_k|\leq \gamma_0.
$$
\label{lem:Linfty_Lipschitz_estimate}
\end{lem}
\begin{proof}
We write $\beta=\Delta t/\Delta x^2$ and abbreviate $\sigma_k^n:=\sigma(v_k^n)$. Arguing by induction on $n$ our statement holds for $n=0$ by assumption on the initial datum.\\
{\bf Step 1: positivity and $l^{\infty}$ stability.} 
Noting that $\frac{v_{k+1}^n-v_{k-1}^n}{2\Delta x}=\frac{w_{k+1}^n+w_{k}^n}{2}$ it is easy to rewrite \eqref{eq:scheme} inside the support $x_k\in[x_{K_l(n)},x_{K_r(N)}]$ as
$$
v_k^{n+1} =	(1-2a)v_k^n +	(a-b)v_{k-1}^n +(a+b)v_{k-1}^n
$$
with
\begin{equation}
a:=\beta(\sigma_k^n+\eps)
\quad\mbox{and}\quad
b:=\beta \Delta x(w_{k+1}^n+w_{k}^n)/4 .
\label{eq:def_ab_Linfty_Lipschitz}
\end{equation}
By the induction hypothesis and monotonicity of $\sigma$ the \eqref{eq:CFL} condition implies
$$
0\leq \beta \eps  \leq a\leq \beta(\sigma(M)+\eps)\leq \frac{1}{2},\qquad |b|\leq \Delta x\beta\gamma_0/2\leq \beta\eps \leq a,
$$
thus $v_k^{n+1}$ is a convex combination of $v_{k-1}^n,v_{k}^n,v_{k+1}^n\in [0,M]$. In particular $0\leq v_{k}^{n+1}\leq M$ for $k\in[K_l(n),K_r(n)]$, and by \eqref{eq:def_linear_interpolation} clearly $0\leq v^{n+1}_k\leq M$ everywhere.
%

%
\noindent{\bf Step 2: Lischitz bounds in the support.} Consider any $k\in[K_l(n)+1, K_r(n)]$, so that $v^{n+1}_k,v^{n+1}_{k-1}$ are both computed using \eqref{eq:scheme}, which we recast in the form
\begin{equation}
v_{k}^{n+1}=v_k^n+\beta \Delta x\left(\sigma_k^n+\eps\right)(w^n_{k+1}-w^n_{k})+\beta \Delta x^2/4\left(w^n_{k+1}+w_{k-1}^n\right)^2.
\label{eq:finite_diff_v_1}
\end{equation}
Subtracting the corresponding equation for $v_{k-1}^{n+1}$ and dividing by $\Delta x$,
straightforward manipulations lead to
\begin{align}
  w_k^{n+1}&=   w_k^n +\Delta t \Bigg[\left(\frac{\sigma_k^n+\sigma_{k-1}^n}{2}+\eps\right) \frac{ w_{k+1}^n-2w_{k}^n+w_{k-1}^n}{\Delta x^2}\nonumber\\
  & \hspace{2cm}+ \left(\mathfrak{S}_{v} w_k^n+ 2 \frac{w_{k+1}^n+2w_{k}^n+w_{k-1}^n}{4}\right)\frac{w_{k+1}^n-w_{k-1}^n}{2 \Delta x}\Bigg]
  \label{eq:diff_wk}
\end{align}
with
$$
\mathfrak{S}_v:=\frac{\sigma_k^n-\sigma_{k-1}^n}{v_k^n-v_{k-1}^n} = \frac{\sigma(v_k^n)-\sigma(v_{k-1}^n)}{v_k^n-v_{k-1}^n}\approx \sigma'(v(x_k,t^n)).
$$
Formula \eqref{eq:diff_wk} is the discrete equivalent of
\begin{equation}
 w=\partial_x v:\qquad \partial_t w=\sigma(v)\partial_{xx}^2w+\big[\sigma'(v)w+2w\big]\partial_x w,
 \label{eq:PDE_w}
\end{equation}
which is formally obtained differentiating \eqref{eq:GPME_v} w.r.t. $x$. Considering \eqref{eq:PDE_w} as a linear parabolic equation $\partial_t w=a\partial_{xx}^2v+b\partial_xw$ with no zero-th order coefficient, we see that $w=\partial_x v$ formally satisfies the maximum principle. Thus the initial $\gamma_0$-Lipschitz bounds for $v^0$ should be preserved for $t\geq 0$ as in our statement.

In order to make this maximum principle rigorous at the discrete level we rewrite \eqref{eq:diff_wk} as
\begin{equation}
w_k^{n+1}=(1-2a)w_k^n+(a-b)w_{k-1}^n+(a+b)w_{k+1}^n,
\label{eq:w_k^n+1}
\end{equation}
with now
\begin{equation}
a=\beta\left(\frac{\sigma_k^n+\sigma_{k-1}^n}{2}+\eps\right)
\quad \mbox{and} \quad 
b=\beta \Delta x\left(\frac{\mathfrak{S}_v w_k^n}{2}+\frac{w_{k+1}^n+2w_{k}^n+w_{k-1}^n}{4}\right).
\label{eq:def_ab}
\end{equation}
By the induction hypothesis $0\leq v_k^n\leq M$ and the structural assumption \eqref{eq:structural_condition_sigma} we get $0\leq \mathfrak{S}_v\approx \sigma'(v_k^n)\leq S_1(M)$, and the \eqref{eq:CFL} condition implies
$$
0\leq\beta \eps\leq  a \leq \beta\Big(\sigma(M)+\eps\Big)\leq  1/2\quad \mbox{and} \quad |b|\leq \beta \Delta x \gamma_0\left(\frac{S_1(M)}{2}+1\right) \leq \beta \eps \leq a.
$$
From \eqref{eq:w_k^n+1} we see that $w_{k}^{n+1}$ is a convex combination of $w_{k-1}^{n},w_{k-1}^{n},w_{k+1}^{n}$ and we conclude that $|w_{k}^{n+1}|\leq \gamma_0$ as claimed.\\
{\bf Step 3: Lischitz bounds close to the interfaces.} The computations at the left and right interfaces are identical, so we only deal with the right one and write $K=K_r(n)$, $s^n=\zeta_r^n-x_{K}$ and $(s^n)'=\zeta_r^{n+1}-x_{K}$ for simplicity. By construction of the scheme $v_k^{n+1}$ is linear for $x_k\in[x_{K},\zeta^{n+1}]$ and zero for $x_k\geq \zeta^{n+1}$. In particular $w_{K+1}^{n+1}\leq w_k^{n+1}\leq 0$ for all $x_k\geq x_{K+1}$ and it is clearly enough to estimate $|w_{K+1}^{n+1}|$. From \eqref{eq:def_linear_interpolation} we see that $w^{n+1}_{K+1}=-\frac{v_K^{n+1}}{(s^n)'}$, and exploiting \eqref{eq:scheme} we get
\begin{align*}
w_{K+1}^{n+1}	
& =-\frac{1}{(s^n)'}\Bigg{[}v_K^n + \Delta t\left(\sigma_K^n+\eps\right)\frac{v_{K+1}^n-2v_K^n+v_{K_1}^n}{\Delta x^2}+\Delta t\left(\frac{v_{K+1}^n-v_{K-1}^n}{2\Delta x}\right)^2 \Bigg{]}\nonumber\\
& =-\frac{1}{(s^n)'}\Bigg{[}v_K^n + \beta\Delta x \left(\sigma_K^n+\eps\right)(w_{K+1}^n-w_K^n)\nonumber\\
  & \hspace{2cm} +\Delta t \left\{(w_{K+1}^n)^2 -\frac{w_K^n+3w_{K+1}^n}{4}(w_{K+1}^n-w_K^n)\right\} \Bigg{]}.
\end{align*}
According to Remark~\ref{rmk:propagation_law_interpolation} we have $(s^n)'-s^n=\zeta^{n+1}-\zeta^n=-w_{K+1}^n\Delta t$, and since $v_K^{n}=-w_{K+1}^ns^n$ we get $v_K^n+\Delta t (w_{K+1}^n)^2=-w_{K+1}^n (s^n)'$. Substituting in the previous expression gives
\begin{equation}
w_{K+1}^{n+1}=w_{K+1}^n+c(w_{K}^n-w_{K+1}^n)
\label{eq:w_k^n+1_interface}
\end{equation}
with
\begin{equation}
c=\frac{\beta \Delta x}{(s^n)'}\left((\sigma_K^n+\eps)-\Delta x\frac{3w_{K+1}^n+w_{K}^n }{4}\right).
\label{eq:def_c}
\end{equation}
Using the induction hypothesis, the \eqref{eq:CFL} condition, and $(s^n)'\geq s^n \geq \Delta x$ yields
$$
0\leq \beta\frac{\Delta x}{(s^n)'}\left(\eps-\gamma_0 \Delta x\right)\leq c\leq \beta\left(\sigma(M)+\eps+\gamma_0\Delta x\right)\leq 1,
$$
thus by \eqref{eq:w_k^n+1_interface} $|w_{K+1}^{n+1}|\leq \gamma_0$ as the convex combination of $w_{K}^{n},w_{K+1}^{n}$ and the proof is complete.
\end{proof}
As a consequence the interfaces propagate with finite speed:
\begin{lem}
For all $t^n\geq 0$ there holds $\left|\frac{\zeta^{n+1}_{lr}-\zeta^n_{lr}}{\Delta t}\right|\leq \gamma_0$ and
$$
\zeta_l(0)-\gamma_0t^n\leq \zeta^n_l\leq \zeta_l(0)\leq \zeta_r(0)\leq \zeta_r^n\leq \zeta_r(0)+\gamma_0t^n.
$$
\label{lem:discrete_speed_interfaces}
\end{lem}
\begin{proof}
By Remark~\ref{rmk:propagation_law_interpolation} $|(\zeta^{n+1}-\zeta^n)/\Delta t|=|-w_{K(n)\pm1}^n|$ so our statement immediately follows by Lemma~\ref{lem:Linfty_Lipschitz_estimate} and the pinning $\zeta_{ln}^0=\zeta_{lr}(0)$. The monotonicity is a consequence of \eqref{eq:propagation_interfaces} with $v_k^n\geq 0$.
\end{proof}

In the next auxiliary lemma we construct the lower bound to be used in the generalized Aronson-B\'enilan estimate $\partial^2_{xx} v\geq \underline{z}(t)$ by means of a certain ODE:
\begin{lem}
Let $\Lambda:=\gamma_0^2S_2(M)$ and $F(z):=\Lambda z +(2+s_1(M))z^2$ with $s_1,S_2$ as in \eqref{eq:structural_condition_sigma}. There is a function $\underline{z}(t):\R^+\to \R$ such that $\frac{d\underline{z}}{dt}=F(\underline{z})$ with $\lim\limits_{t\searrow 0}\underline{z}(t)=-\infty$. Moreover $\underline{z}$ is monotone increasing and concave, $\underline{z}(t)\leq \underline{z}(\infty)=-\Lambda/(2+s_1(M))$, and $\underline{z}(t)\sim -\frac{1}{(2+s_1(M))t}$ when $t\searrow 0$.
\label{lem:def_ODE_z}
\end{lem}
\begin{proof}
Observe that $F(z)$ is a quadratic polynomial with $F(-\Lambda/(2+s_1(M)))=0$. Picking any $t_0>0,z_0<-\Lambda/(2+s_1(M))$ and solving $dz/dt=F(z)$ with $z(t_0)=z_0$ it is easy to see that $z$ is monotone increasing in $(\underline{T},\infty)$ with blow-up in finite time $z(\underline{T})=-\infty$ and $z(\infty)=-\Lambda/(2+s_1(M))$. Shifting $\underline{z}(t):=z(t+\underline{T})$ gives the sought solution, and all the qualitative properties follow from a straightforward phase portrait analysis.
\end{proof}
The generalized Aronson-B\'enilan estimate then takes the form
\begin{lem}
Let $\underline{z}(t)$ as in Lemma~\ref{lem:def_ODE_z}. Then for all $k,n$ there holds
\begin{equation}
Z_k^n:=\frac{Av_k^n}{\Delta x^2}=\frac{v_{k-1}^n-2v_{k}^n+v_{k+1}^n}{\Delta x^2}\geq \underline{z}(t^n).
\label{eq:Aronson_Benilan}
\end{equation}
\label{lem:Aronson_Benilan_estimate}
\end{lem}
\begin{proof}
Since $\underline{z}$ is monotone increasing with $\underline{z}(0)=-\infty$ the time $t^N=\max\{t^n:\,\underline{z}(t^n)\leq -2\gamma_0/\Delta x\}$ is well defined and positive, provided that $\Delta x,\Delta t$ are small enough. By Lemma~\ref{lem:Linfty_Lipschitz_estimate} we have $Z_{k}^n=\frac{w_{k+1}^n-w_{k}^n}{\Delta x}\geq -2\gamma_0/\Delta x$ and our estimate automatically holds if $t_n\leq t^N$. We argue now by induction on $n\geq N$.\\
{\bf Step 1: estimate in the support.}  Consider first any $k\in[K_l(n)+1,K_r(n)-1]$, so that $v_{k-1}^{n+1},v_k^{n+1},v_{k+1}^{n+1}$ are all computed from the finite difference equation \eqref{eq:scheme}. Applying the second order difference operator $A$ to \eqref{eq:scheme} and dividing by $\Delta x^2$, straightforward algebra leads to
\begin{align}
Z_{k}^{n+1}	=&  Z_k^n + \Delta t \Bigg{[}(\mathfrak{S}+\eps)\frac{AZ_{k}^n}{\Delta x^2}
+2 \left(\mathfrak{S}_x+W_{1}\right)\left(\frac{Z_{k+1}^n-Z_{k-1}^n}{2\Delta x}\right)
\nonumber\\
& \hspace{3cm}
+ \mathfrak{S}_{vv}(W_{2})^2\frac{Z_{k-1}^n+2Z_k^n+Z_{k+1}^n}{4}
\nonumber\\
   &\hspace{2cm}
   +\Bigg{\{}\mathfrak{S}_v Z_k^n\frac{Z_{k-1}^n+2Z_k^n+Z_{k+1}^n}{4} 
   + 2 \left(\frac{Z_{k-1}^n+2Z_k^n+Z_{k+1}^n}{4}\right)^2 \Bigg{\}}\Bigg{]}
 \label{eq:discrete_PDE_Z}
\end{align}
with
\begin{align*}
& \mathfrak{S}	:=  \frac{\sigma_{k-1}^n+2\sigma_{k}^n+\sigma_{k+1}^n}{4}\approx \sigma(v(x_k,t^n)),
\qquad
\mathfrak{S}_x := \frac{\sigma_{k+1}^n-\sigma_{k-1}^n}{2\Delta x}\approx \partial_x \sigma(v(x_k,t^n)),\\
& \mathfrak{S}_v := \frac{1}{2}\left(\frac{\sigma_{k+1}^{n}-\sigma_{k}^{n}}{v_{k+1}^{n}-v_{k}^{n}} + \frac{\sigma_{k}^{n}-\sigma_{k-1}^{n}}{v_{k}^{n}-v_{k-1}^{n}}\right)\approx \sigma'(v(x_k,t^n)),\\
&\mathfrak{S}_{vv}	 :=2\frac{(v_k^n-v_{k-1}^n)\sigma_{k+1}^n-(v_{k+1}^n-v_{k-1}^n)\sigma_{k}^n + (v_{k+1}^n-v_{k}^n)\sigma_{k-1}^n}{(v_{k+1}^n-v_{k}^n)(v_{k}^n-v_{k-1}^n)(v_{k+1}^n-v_{k-1}^n)}\approx \sigma''(v(x_k,t^n)),
\end{align*}
and
$$
W_{1} :=\frac{\overline{w}_{k-1}^n+2\overline{w}_{k}^n+\overline{w}_{k+1}^n}{4}\approx\partial_xv(x_k,t^n),
\qquad
W_{2} :=\overline{w}_{k}^n\approx\partial_xv(x_k,t^n)
$$
(recall that we write $\sigma_k^n=\sigma(v_k^n)$ and $\overline{w}_{k}^n=(v_{k+1}^n-v_{k-1}^n)/2\Delta x$). Note that \eqref{eq:discrete_PDE_Z} is nothing but the discrete equivalent of
\begin{equation}
\partial_t z=\sigma(v)\partial_{xx}^2 z
+2\Big{[}\partial_x \sigma(v)+\partial_x v\Big{]}\partial_x z
+\Big{[} \sigma''(v)|\partial_x v|^2\Big{]}z
+\Big{[}\sigma'(v)+2\Big{]}z^2
\label{eq:PDE_z=vxx}
\end{equation}
for $z=\partial_{xx}^2 v$, which is obtained differentiating twice $\partial_t v=\sigma(v)\partial_{xx}^2v+|\partial_x v|^2$ w.r.t. $x$. Let us give a formal proof that $z=\partial^2_{xx}v\geq \underline{z}(t)$ at the continuous level: since $0\leq v(x,t)\leq M$ we have $0<s_1(M)\leq \sigma'(v)$ and $|\sigma''(v)|\leq S_2(M)$, and recall that $|\partial_x v|\leq \gamma_0$. Using the definition of $\underline{z}(t)$ in Lemma~\ref{lem:def_ODE_z} it is easy to check that $\underline{z}(t)$ is a subsolution of \eqref{eq:PDE_z=vxx}. Since $\underline{z}(0)=-\infty\leq z(x,0)$ the comparison principle should give $z(x,t)\geq\underline{z}(t)$.
In order to reproduce this formal computation at the discrete level let us first rewrite \eqref{eq:discrete_PDE_Z} as
\begin{align}
Z_{k}^{n+1}=	&	\left[1-2\beta(\mathfrak{S}+\eps)   +   \frac{\beta \Delta x^2(W_{2})^2\mathfrak{S}_{vv}}{2}\right]Z_k^{n}
\nonumber\\
&	+\beta\left[(\mathfrak{S}+\eps)	+   \frac{\Delta x^2(W_{2})^2\mathfrak{S}_{vv}}{4}  -  \Delta x(\mathfrak{S}_x+W_{1})\right]Z_{k-1}^{n}
\nonumber\\
&	+\beta\left[(\mathfrak{S}+\eps)	+   \frac{\Delta x^2(W_{2})^2\mathfrak{S}_{vv}}{4}  +  \Delta x(\mathfrak{S}_x+W_{1})\right]Z_{k-1}^{n}
\nonumber\\
  & +\frac{\beta \Delta x^2}{8}\left(Z_{k-1}^n+2Z_{k}^n+Z_{k+1}^n\right)\left[Z_{k-1}^n+2(1+\mathfrak{S}_v)Z_{k}^n+Z_{k+1}^n\right].
  \label{eq:discrete_Zk^n+1}
\end{align}
We show now \eqref{eq:discrete_Zk^n+1} satisfies the discrete comparison principle, in the sense that $Z_{k}^{n+1}$ is non-decreasing in the three arguments $Z_{k-1}^n,Z_{k}^n,Z_{k+1}^n$. To this end we first note that
$$
|Z_{k}^n|=|(w_{k+1}^n-w_{k}^n)/\Delta x|\leq 2\gamma_0/\Delta x,
$$
and by our structural hypotheses \eqref{eq:structural_condition_sigma} and Lemma~\ref{lem:Linfty_Lipschitz_estimate} it is easy to check that
$$
\begin{array}{ccc}
0\leq \mathfrak{S}\leq \sigma(M)
&
|\mathfrak{S}_x|\leq S_1(M)\gamma_0,
&
0\leq \mathfrak{S}_v\leq S_1(M),\\

|\mathfrak{S}_{vv}|\leq S_2(M),
&
|W_{1}|\leq \gamma_0,
&
|W_{2}|\leq \gamma_0.
\end{array}
$$
Thus by the \eqref{eq:CFL} condition
\begin{align*}
\frac{\partial Z_{k}^{n+1}}{\partial Z_k^n} & =  1 -2\beta (\mathfrak{S}+\eps)+\frac{\beta \Delta x^2(W_{2})^2\mathfrak{S}_{vv}}{2}\\
  & +\frac{\beta \Delta x^2}{8}\Big{[}2\Big{(}Z_{k-1}^n+2(1+\mathfrak{S}_v)Z_{k}^n+Z_{k+1}^n\Big{)}
  + 2(1+\mathfrak{S}_v)\Big{(}Z_{k-1}^n+2Z_{k}^n+Z_{k+1}^n\Big{)}
  \Big{]}\\
  & \geq 1-\beta\left[2\Big(\sigma(M)+\eps\Big)+\frac{\Delta x^2 \gamma_0^2 S_2(M)}{2}+\gamma_0\Delta x\Big(4+3S_1(M)\Big)\right]\geq 0,
\end{align*}
\begin{align*}
\frac{\partial Z_{k}^{n+1}}{\partial Z_{k-1}^n} & = \beta\Bigg{[}(\mathfrak{S}+\eps)	+   \frac{\Delta x^2(W_{2})^2\mathfrak{S}_{vv}}{4}  -  \Delta x(\mathfrak{S}_x+W_{1})\\
  & \hspace{1cm}+\frac{\Delta x^2}{8}\Big\{(Z_{k-1}^n+2Z_{k}^n+Z_{k+1}^n)+(Z_{k-1}^n+2(1+\mathfrak{S}_v)Z_{k}^n+Z_{k+1}^n))\Big\}\Bigg{]}\\
  & \geq \beta\left[\eps-\gamma_0\Delta x\left\{\frac{\Delta x\gamma_0S_2(M)}{4}+\Big(S_1(M)+1\Big)+\frac{1}{2}\Big(4+S_1(M)\Big)\right\}\right]\geq 0,
\end{align*}
and similarly $\frac{\partial Z_{k}^{n+1}}{\partial Z_{k+1}^n}\geq 0$. By the induction hypothesis we see that $Z_{k}^{n+1}$ is greater or equal to the right-hand side of \eqref{eq:discrete_Zk^n+1} evaluated with $Z_{k-1}^n,Z_{k}^n,Z_{k+1}^n\geq \underline{z}(t^n)$, and using the structural assumptions $\mathfrak{S}_v\geq s_1(M)$ and $|\mathfrak{S}_{vv}|\leq S_2(M)$ we get
\begin{align*}
Z_{k}^{n+1}& \geq  \underline{z}(t^n)+ \beta \Delta x^2(W_{2})^2\mathfrak{S}_{vv}\,\underline{z}(t^n)+\beta \Delta x^2(2+\mathfrak{S}_v)\underline{z}^2(t^n)
\\
  & \geq \underline{z}(t^n)+ \Delta t\left[\gamma_0^2S_2(M)\underline{z}(t^n)+(2+s_1(M))\underline{z}^2(t^n)\right].
\end{align*}
In the righ-hand side we recognize $\underline{z}(t^n)+\Delta t F(\underline{z}(t^n))$ with $F$ as in Lemma~\ref{lem:def_ODE_z}. Since by construction $\dot{\underline{z}}=F(\underline{z})$ and $\underline{z}$ is concave we finally get
$$
Z_{k}^{n+1}\geq \underline{z}(t^n)+\dot{\underline{z}}(t^n)[t^{n+1}-t^n]\geq \underline{z}(t^{n+1})
$$
as required.\\
{\bf Step 2: estimate close to the interfaces.}
We only establish the AB estimate across the right interface and boundary layer, and write again $\zeta=\zeta_r$ and $K=K_r(n)$ to keep the notations light (the argument is identical to the left). Recall that for $x_k\in[x_{K},\zeta^{n+1}]$ the next step $v_{k}^{n+1}\geq 0$ is linearly interpolated by \eqref{eq:def_linear_interpolation}, and $v_{k}^{n+1}=0$ for $x_k\geq \zeta^{n+1}$. As a consequence $Av_k^{n+1}\geq 0$ for $k>K$ and \eqref{eq:Aronson_Benilan} is trivially satisfied there as $Z_k^{n+1}\geq 0>\underline{z}(t^{n+1})$. Hence we only need to look at $k=K$.

By definition of $K=K_r(n)$ we see that $w_{K}^{n+1}$, $w_{K+1}^{n+1}$ satisfy \eqref{eq:w_k^n+1} and \eqref{eq:w_k^n+1_interface}, namely
$$
w_K^{n+1}=w_K^{n}+(a+b)\Delta x Z_K^{n}-(a-b)\Delta x Z_{K+1}^{n}
\quad\mbox{and}\quad
w_{K+1}^{n+1}=w_{K+1}^n-c\Delta xZ_K^n
$$
with $a,b$ as in \eqref{eq:def_ab} with $k=K$ and $c$ as in \eqref{eq:def_c}. Subtracting and dividing by $\Delta x$ we get that $Z_{K}^{n+1}=(w_{K+1}^{n+1}-w_{K}^{n+1})/\Delta x$ can be expressed as
\begin{equation}
Z_{K}^{n+1}=(1-a-b-c)Z_{K}^n + (a-b)Z_{K-1}^n.
\label{eq:Z_K^n+1_interface}
\end{equation}
We claim as in step 1 that the right-hand side is nondecreasing in $Z_{K}^n,Z_{K-1}^n$. Indeed we already showed in the proof of Lemma~\ref{lem:Linfty_Lipschitz_estimate} that $a-|b|\geq 0$, and recalling that $(s^n)'=\zeta^{n+1}-x_{K(n)}\geq\zeta^{n}-x_{K(n)} \geq \Delta x$ we compute
\begin{align*}
1-a-b-c	&= 1 - \beta\left(\frac{\sigma_K^n+\sigma_{K-1}^n}{2}+\eps\right)\\
  & \hspace{1cm}-\beta \Delta x\left(\frac{1}{2}\frac{\sigma_K^n-\sigma_{K-1}^n}{v_K^n-v_{K-1}^n} w_K^n+\frac{w_{K-1}^n+2w_{K}^n+w_{K+1}^n}{4}\right)\\
&\hspace{2cm}-\frac{\beta \Delta x}{s_n'}\left((\sigma_K^n+\eps)-\Delta x\frac{3w_{K+1}^n+w_{K}^n}{4}\right)\\
  & \geq 1-\beta\left[
  \Big(\sigma(M)+\eps\Big) + \Delta x\left(\frac{S_1(M)\gamma_0}{2}+\gamma_0\right)
  + \Big(\sigma(M)+\eps+\gamma_0\Delta x\Big)
  \right]\geq 0,
\end{align*}
where the last inequality follows by the \eqref{eq:CFL} condition. Before evaluating \eqref{eq:Z_K^n+1_interface} with $Z_{K}^n,Z_{K-1}^n\geq \underline{z}(t^n)$ we first recall from Lemma~\ref{lem:discrete_speed_interfaces} that the interfaces propagate with discrete speed at most $\gamma_0$, and that by the CFL condition $\Delta t=\mathcal{O}(\Delta x^2)$. In particular $\Delta x \leq s_n\leq  (s^n)'=s_n+(\zeta^{n+1}-\zeta^n)\leq 2\Delta x +\mathcal{O}( \Delta x^2)$, thus $\Delta x\leq (s^n)'\leq 3\Delta x$ for small $\Delta x$ and
\begin{align*}
1-2b-c	& = 1-\beta \Delta x\left(\frac{\sigma_K^n-\sigma_{K-1}^n}{v_K^n-v_{K-1}^n}w_K + \frac{w_{K-1}^n+2w_{K}^n+w_{K+1}^n}{2}\right)\\
  & \hspace{3cm}-\frac{\beta \Delta x}{(s^n)'}\left((\sigma_K^n+\eps)-\Delta x\frac{3w_{K+1}^n+w_{K}^n}{4}\right)\\
  & \leq 1 + \beta\gamma_0\Delta x\big( S_1(M)+2  \big)-\frac{\beta \eps}{3} +\beta \gamma_0 \Delta x\\
  & \leq 1-\Delta t\big(2+s_1(M)\big)\frac{3\gamma_0}{\Delta x}
\end{align*}
by the \eqref{eq:CFL} condition ($\eps\geq c\gamma_0\Delta x$). For small $\Delta x,\Delta t$ and by definition of $t^N=\max\{t^{n'}:\,F(t^{n'})\leq -2\gamma_0/\Delta x\}$ it is easy to check that $F(t^N)\sim -2\gamma_0/\Delta x$, and because $\underline{z}$ is increasing and our induction is on $n\geq N$ we can assume that $-3\gamma_0/\Delta x<-2\gamma_0/\Delta x\approx \underline{z}(t^N)\leq \underline{z}(t^n)$ hence
$$
1-2b-c \leq 1+\Delta t\big(2+s_1(M)\big)\underline{z}(t^*)\leq 1+\Delta t\big(2+s_1(M)\big)\underline{z}(t^n).
$$
Evaluating \eqref{eq:Z_K^n+1_interface} with $Z_{k-1}^n,Z_{k}^n,Z_{k+1}^n\geq \underline{z}(t^n)$ thus gives
\begin{align*}
Z_{K}^{n+1}\geq (1-2b-c)\underline{z}(t^n)
  &	\geq \underline{z}(t^{n})+\Delta t\big(2+s_1(M)\big)\underline{z}^2(t^n)\\
  & \geq	\underline{z}(t^{n})+\Delta t\Big[\Lambda \underline{z}(t^n)  +  \big(2+s_1(M)\big)\underline{z}^2(t^n)\Big]\\
  & =\underline{z}(t^{n})+\Delta t F(\underline{z}(t^n)),
\end{align*}
and we conclude by concavity of $\underline{z}$ as in step 1.
%
\end{proof}
\begin{rmk}
For the pure PME nonlinearity $\Phi(s)=s^m$ one has $\sigma(r)=(m-1)r$ and therefore $s_1(M)=s_1=(m-1)$ and  $S_2(M)=0$ in \eqref{eq:structural_condition_sigma}. Tthe ODE for $\underline{z}$ then becomes $\dot{z}=(m+1)z^2$, thus $\underline{z}(t)=-1/(m+1)t$ in Lemma~\ref{lem:def_ODE_z} and we recover the optimal Aronson-B\'enilan estimate $\partial^2_{xx} v\geq -1/(m+1)t$. For general nonlinearities the optimal estimate \cite{CP82} takes the form $\partial_{xx}^2 v\geq -h(v)/t$ for some structural function $h$ related to $\Phi$. Unfortunately we were not able to reproduce the optimal computations at the discrete level, and we shall be content here with our lower bound $\partial_{xx}^2 v\geq \underline{z}(t)\sim-C(1+1/t)$.
\end{rmk}
%
%
\begin{lem}
There is $C=C(v^0)>0$ only such that
$$
\forall n\geq 0,\,k\notin [K_l(n),K_r(n)]:\qquad \left|\frac{v_k^{n+1}-v_k^n}{\Delta t}\right|\leq C.
$$
\label{lem:dv/dt_interface}
\end{lem}
\begin{proof}
The argument is identical to \cite[Lemma 2.4]{DBH84}.
\end{proof}
Combining Lemma~\ref{lem:Aronson_Benilan_estimate} and Lemma~\ref{lem:dv/dt_interface} we get
\begin{cor}
There is $C=C(v^0)>0$ such that
\begin{equation}
\sum\limits_k\left|\frac{Av_k^n}{\Delta x^2}\right|\Delta x +\sum\limits_k\left|\frac{v_{k}^{n+1}-v^n_k}{\Delta t}\right|\Delta x\leq C\left(1+\frac{1}{t^n}+T\right)
\label{eq:L1-bounds_vxx_vt}
\end{equation}
for all $t^n\leq T$.
\label{cor:L1-bounds_vxx_vt}
\end{cor}
\begin{proof}
From Lemma~\ref{lem:Aronson_Benilan_estimate} and $-C(1+1/t)\leq \underline{z}(t)\leq 0$ we see that
$$
\left|\frac{Av_k^n}{\Delta x^2}\right|	\leq	\frac{Av_k^n}{\Delta x^2}+2|\underline{z}(t^n)|\leq \frac{Av_k^n}{\Delta x^2} +C\left(1+\frac{1}{t^n}\right).
$$
Multiplying by $\Delta x$ and summing over $k$'s with $v^n_k=0$ outside an interval of length $C(1+t^n)$ (Lemma~\ref{lem:discrete_speed_interfaces}) we get the first part of the estimate
$$
\sum\limits_k\left|\frac{Av_k^n}{\Delta x^2}\right|\Delta x\leq C\left(1+\frac{1}{t^n}\right)(1+t^n)\leq C(1+1/t^n+T).
$$
Inside the support $k\in[K_l(n),K_r(n)]$ the time derivative can be estimated from \eqref{eq:scheme} as
\begin{align*}
\left|\frac{v_k^{n+1}-v_k^n}{\Delta t}\right| & =\left|\Big(\sigma(v_k^n)+\eps\Big)\frac{Av_k^n}{\Delta x^2}+\left(\frac{v_{k+1}^n-v_{k-1}^n}{2\Delta x}\right)^2\right|\\
  & \leq \Big(\sigma(M)+\eps\Big)\left|\frac{Av_k^n}{\Delta x^2}\right|+\gamma_0^2\leq C\left(\left|\frac{Av_k^n}{\Delta x^2}\right|+1\right),
\end{align*}
and inside the boundary layers of thickness $\Delta x\leq s^n\leq 2\Delta x$ the time derivative is estimated by Lemma~\ref{lem:dv/dt_interface}. Multiplying by $\Delta x$ and summing over $k$'s as before gives the second part of the estimate and ends the proof.
\end{proof}
%
We end this section with uniform H\"oder estimates in time up to $t=0^+$, which are inherited from the initial Lipschitz regularity for $v^0(x)$.
\begin{prop}
For any $T>0$ there is $C=C(T,v^0)>0$ such that
$$
|v_{k}^n-v_k^m|\leq C|t^n-t^m|^{1/2}
$$
for all $t^n,t^m\in[0,T]$.
\label{prop:Holder}
\end{prop}
The proof is almost identical to \cite[Lemma 2.7]{DBH84}, and the argument is a discrete version of that in \cite{G76}. However we will need to make sure in Section~\ref{section:two_patches} that the proof carries out for the hole-filling problem so we give nonetheless the full details for the sake of completeness.
\begin{proof}
We argue locally in cylinders
$$
Q=[x_{k_0}-r,x_{k_0}+r]\times[t^{n_0},t^{n_1}],
$$
where $x_{k_0}$ and $0\leq t^{n_0}\leq t^{n_1}\leq T$ are fixed and $r$ is a multiple of $\Delta x$ to be adjusted.\\
{\bf Step 1:} letting
$$
\begin{array}{c}
H:=\max\limits_{n_0\leq n\leq n_1}\,|v_{k_0}^n-v_{k_0}^{n_0}|,\qquad c:= 2\Big(\sigma(M)+\eps\Big)+\gamma_0 r\\
 V_k^n:= v_k^n-v_{k_0}^{n_0}-\gamma_0 r -\frac{H}{r^2}\left[(x_k-x_{k_0})^2+c(t_n-t_{n_0})\right],
 \end{array}
$$
we claim that
\begin{equation}
V_k^n\leq 0\qquad \mbox{for all }(x_k,t^n)\in Q.
\label{eq:V_k^n_leq_0}
\end{equation}
Arguing by induction on $n$, \eqref{eq:V_k^n_leq_0} holds for $n=n_0$ as $V_{k}^{n_0}\leq \left|v_{k}^{n_0}-v_{k_0}^{n_0}\right|-\gamma_0r\leq 0$ since $|x_k-x_{k_0}|\leq r$ and $|w^n_k|\leq \gamma_0$. For the induction step we consider three cases: (i) $x_k\notin[\zeta_l^{n+1},\zeta_r^{n+1}]$, (ii) $x_k$ is inside the boundary layer, and (iii) $x_k$ is inside the numerical support where \eqref{eq:scheme} holds.

In the first case we have $v_k^{n+1}=0$ and our claim immediately holds by definition of $V_{k}^{n+1}$ with $v_{k_0}^{n_0}\geq 0$. For (ii) we have $x_k\in[x_{K(n)},\zeta^{n+1}]$, and we have already shown that $(s^n)'=|\zeta^{n+1}-x_{K(n)}|\leq 3\Delta x$ for small $\Delta x$. In the boundary layer $v_{k}^{n+1}$ is computed by linear interpolation with slope $|w_k^{n+1}|\leq \gamma_0$ and therefore
$$
V_{k}^{n+1}\leq v_{k}^{n+1}-\gamma_0\leq \gamma_0 .3\Delta x-\gamma_0 r\leq 0
$$
provided that $r\geq 3\Delta x$, which will be ensured in step 2. In the last case (iii) we consider the linearized operator $L$ of \eqref{eq:scheme}, whose action on any sequence $a_k^n$ is defined as
$$
La_k^{n+1}:=\frac{a_k^{n+1}-a_k^n}{\Delta t}-\Big(\sigma(v_k^n)+\eps\Big)\frac{A a_k^n}{\Delta x^2}-\left(\frac{v_{k+1}^n-v_{k-1}^n}{2\Delta x}\right)\left(\frac{a_{k+1}^n-a_{k-1}^n}{2\Delta x}\right).
$$
Applying $L$ to $V_k^{n+1}$ with $Lv_k^{n+1}=0$ as in \eqref{eq:scheme}, it is easy to compute
\begin{align*} 
LV_{k}^{n+1}	& =\frac{H}{r^2}\left[-c + 2\Big(\sigma(v_k^n)+\eps\Big)+\left(\frac{v_{k+1}^n-v_{k-1}^n}{2\Delta x}\right)\left(x_k-x_{k_0}\right)\right] \\
  & \leq \frac{H}{r^2}\left[-c+2\Big(\sigma(M)+\eps\Big)+\gamma_0r\right]\leq 0
\end{align*}
by definition of $c$. The inequality $LV_k^{n+1}\leq 0$ can then be rewritten as
$$
V_{k}^{n+1}\leq (1-2a)V_{k}^n+(a-b)V_{k-1}^n+(a+b)V_{k+1}^{n}
$$
with coefficients $a,b$ exactly as in \eqref{eq:def_ab_Linfty_Lipschitz}. We already showed in the proof of Lemma~\ref{lem:Linfty_Lipschitz_estimate} that $0\leq a \leq 1/2$ and $|b|\leq a$. In particular the above right-hand side is a convex combination of $V_{k-1}^n,V_{k}^n,V_{k+1}^n$, thus $V_{k}^{n+1}\leq 0$ as desired.\\
{\bf Step 2.} Choosing $k=k_0$ in $V_k^n\leq 0$ we see that $v_{k_0}^n-v_{k_0}^{n_0}\leq \gamma_0r+\frac{cH}{r^2}|t^{n_1}-t^{n_0}|$, and in a similar way we get the same upper bound for $v_{k_0}^{n_0}-v_{k_0}^{n}$. Taking the maximum over $n\in [n_0,n_1]$ and writing $s=|t^{n_1}-t^{n_0}|$ we see by definition of $H$ that
\begin{equation}
H\leq \gamma_0 r+\frac{cH}{r^2}s.
\label{eq:H_leq}
\end{equation}
Choose now $r$ to be a multiple of $\Delta x$ such that
$$
r_1+3\Delta x\leq r\leq r_1+4\Delta x,
$$
where $r_1>0$ is the largest root of
$$
\rho^2-2cs=\rho^2-2\gamma_0s\rho+4\Big(\sigma(M)+\eps\Big)s=0.
$$
In particular $3\Delta x \leq r$ as required in step 1, and it is easy to check that $r_1\lesssim Cs^{1/2}$ when $s\leq T$. Moreover $cs/r^2\leq 1/2$ and \eqref{eq:H_leq} give
$$
H/2\leq \gamma_0r \leq \gamma_0(r_1+4\Delta x)\leq C(s^{1/2}+\Delta x).
$$
Now $s=|t^{n_1}-t^{n_0}|$ and $\Delta x =\sqrt{\Delta t/\beta}\leq \beta^{-1/2}|t^{n_1}-t^{n_0}|^{1/2}$, so finally
$$
|v_{k_0}^{n_1}-v_{k_0}^{n_0}|\leq H\leq C|t^{n_1}-t^{n_0}|^{1/2}
$$
and the proof is complete.
\end{proof}
%

%
%
\subsection{Convergence of the approximate solution and interfaces}
\label{subsection:CV_one_patch}
Denoting the mesh parameters $h=(\Delta t,\Delta x)$ and $L_k^n,U_k^n$ the lower and upper triangles in Figure~\ref{fig:FIG2},
\begin{figure}[h!]
 \def\svgwidth{.7\textwidth}
 \begin{center}
\begingroup%
  \makeatletter%
  \providecommand\color[2][]{%
    \errmessage{(Inkscape) Color is used for the text in Inkscape, but the package 'color.sty' is not loaded}%
    \renewcommand\color[2][]{}%
  }%
  \providecommand\transparent[1]{%
    \errmessage{(Inkscape) Transparency is used (non-zero) for the text in Inkscape, but the package 'transparent.sty' is not loaded}%
    \renewcommand\transparent[1]{}%
  }%
  \providecommand\rotatebox[2]{#2}%
  \ifx\svgwidth\undefined%
    \setlength{\unitlength}{841.88974609bp}%
    \ifx\svgscale\undefined%
      \relax%
    \else%
      \setlength{\unitlength}{\unitlength * \real{\svgscale}}%
    \fi%
  \else%
    \setlength{\unitlength}{\svgwidth}%
  \fi%
  \global\let\svgwidth\undefined%
  \global\let\svgscale\undefined%
  \makeatother%
  \begin{picture}(1,0.70707072)%
    \put(0,0){\includegraphics[width=\unitlength]{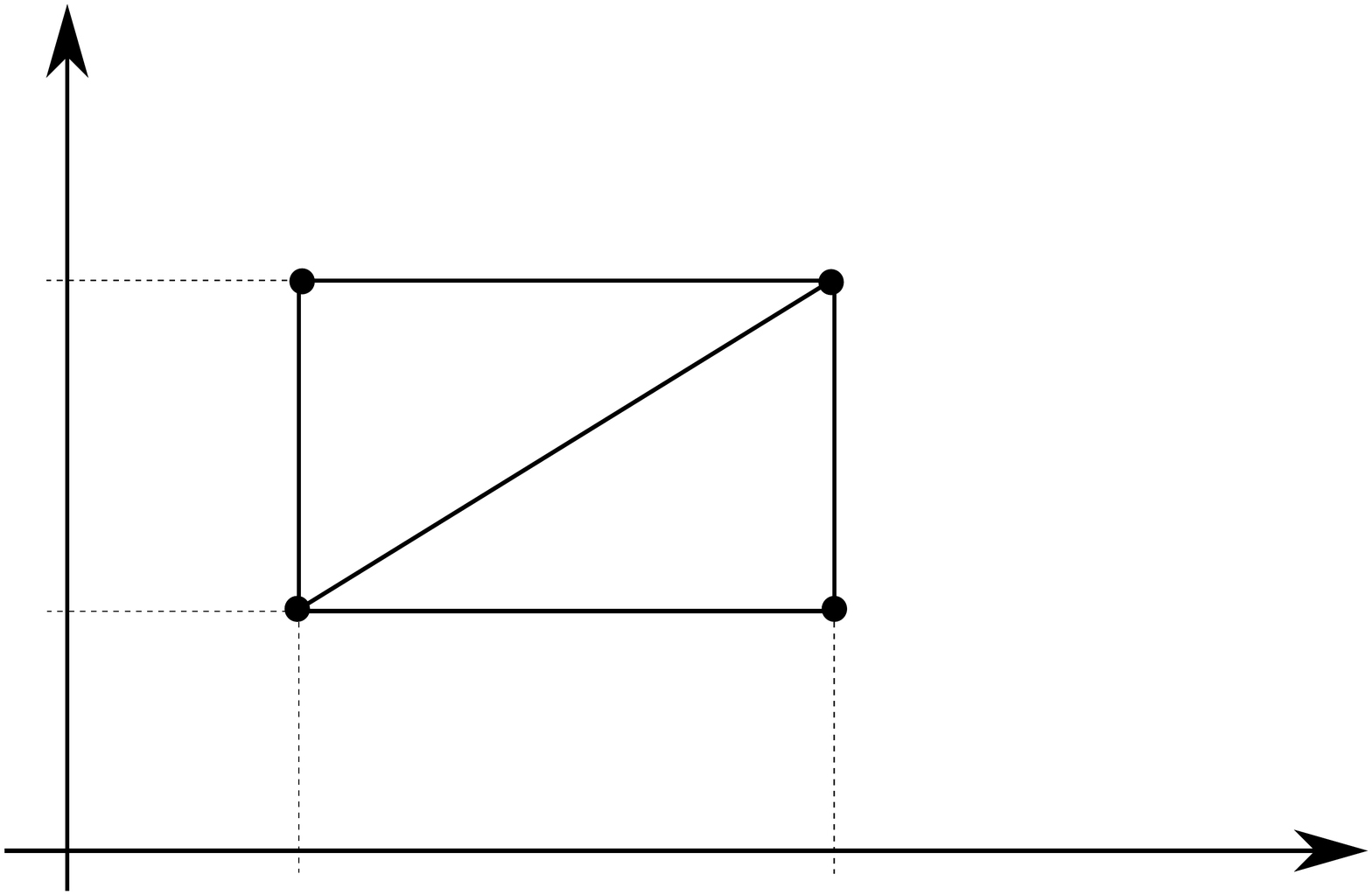}}%
    \put(0.90972295,0.04297841){\color[rgb]{0,0,0}\makebox(0,0)[lb]{\smash{$x$}}}%
    \put(0.02188702,0.24294591){\color[rgb]{0,0,0}\makebox(0,0)[lb]{\smash{$t^n$}}}%
    \put(0.00984916,0.46509914){\color[rgb]{0,0,0}\makebox(0,0)[lb]{\smash{$t^{n+1}$}}}%
    \put(0.02517007,0.62049697){\color[rgb]{0,0,0}\makebox(0,0)[lb]{\smash{$t$}}}%
    \put(0.22134711,0.03939664){\color[rgb]{0,0,0}\makebox(0,0)[lb]{\smash{$x_k$}}}%
    \put(0.57453424,0.04158532){\color[rgb]{0,0,0}\makebox(0,0)[lb]{\smash{$x_{k+1}$}}}%
    \put(0.46291044,0.3097013){\color[rgb]{0,0,0}\makebox(0,0)[lb]{\smash{$L_k^n$}}}%
    \put(0.29766345,0.39396634){\color[rgb]{0,0,0}\makebox(0,0)[lb]{\smash{$U_k^n$}}}%
  \end{picture}%
\endgroup%

 \end{center}
  \caption[Figure~\ref{fig:FIG2}]{linear interpolation domains}
 \label{fig:FIG2}
 \end{figure}
we first define the continuous and piecewise linear interpolation
\begin{equation}
v_h(x,t):=
\left\{
\begin{array}{ll}
v_k^n+(x-x_k)\frac{v_{k+1}^n-v_k^n}{\Delta x}+(t-t^n) \frac{v_{k+1}^{n+1}-v_{k+1}^n}{\Delta t},	&(x,t)\in L_k^n\\
v_{k+1}^{n+1}+(x-x_{k+1})\frac{v_{k+1}^{n+1}-v_k^{n+1}}{\Delta x}+(t-t^{n+1}) \frac{v_{k}^{n+1}-v_{k}^n}{\Delta t},	&(x,t)\in U_k^n
\end{array}
\right..
\label{eq:def_interpolation_vh_meshpoints}
\end{equation}

We also interpolate the interfaces by the piecewise linear curves
\begin{equation}
 \zeta_{h,lr}(t):=  \zeta^n_{lr}+(t-t^n)\frac{\zeta_{lr}^{n+1}-\zeta_{lr}^{n}}{\Delta t}, \qquad t\in [t^n,t^{n+1}].
\label{eq:def_interpolation_interface_meshpoints}
\end{equation}
If $Q_T=\R\times (0,T)$ the estimates from Section~\ref{subsection:estimates_one_patch} can be summarized as
\begin{equation}
0 \leq v_h(x,t)\leq M
\quad\mbox{and}\quad
|\partial_x v_h(x,t)|\leq \gamma_0
\qquad\mbox{a.e. in }Q_T,
\label{eq:estimate_Linfty_Lipschitz_vh}
\end{equation}
\begin{equation}
\forall t_1,t_2\in[0,T]:\qquad |v_h(x,t_1)-v_h(x,t_2)| \leq C(T,v^0)|t_1-t_2|^{1/2},
\label{eq:estimate_Holder_vh}
\end{equation}
\begin{equation}
 \forall\ 0<t\leq T:\qquad
 \int\limits_{\R}\left|\partial_{xx}^2 v_h(\,.\,,t)\right| +\int\limits_{\R}\left|\partial_tv_h(\,.\,,t)\right|\leq C\left(1+\frac{1}{t}+T\right)
 \label{eq:estimate_vxx_vt_measures}
\end{equation}
as measures in $\R$
, and
\begin{equation}
\left|\frac{d\zeta_{h,lr}}{dt}\right|\leq \gamma_0
 \mbox{ and }
 \operatorname{supp}v_h(\,.\,,t)\subseteq [\zeta_l(0)-\Delta x-\gamma_0t,\zeta_r(0)+\gamma_0t+\Delta x]
 \mbox{ for a.e. }t\in [0,T]
 \label{eq:estimate_interface_vh}
\end{equation}
(Lemma~\ref{lem:Linfty_Lipschitz_estimate}, Proposition~\ref{prop:Holder}, Lemma ~\ref{lem:Aronson_Benilan_estimate}, and Lemma~\ref{lem:discrete_speed_interfaces}). The extra $\Delta x$ is needed in \eqref{eq:estimate_interface_vh} because $\zeta^n$ needs not be an integer meshpoint, while $v_h$ is only interpolated from the $(x_k,t^n)$ nodes. It is well known \cite{DK86} that the Cauchy problem \eqref{eq:GPME_v} has a unique solution. As in \cite[Theorem 3.3]{DBH84} the main convergence result then reads:
\begin{theo}
Let $v$ be the unique solution to \eqref{eq:GPME_v} with initial datum $v^0$, and $\zeta_{l,r}$ the corresponding interfaces with $\operatorname{supp}v(\,.\,,t)=[\zeta_l(t),\zeta_r(t)]$. Then
\begin{align}
v_h\to v		\qquad 	&		\mbox{uniformly in }\overline{Q_T},\label{eq:CV_vh_v_one_patch}\\
\partial_x v_h \to \partial_x v	\qquad 	& \mbox{in }L^p(Q_T)\mbox{ for all }p\in[1,\infty),\label{eq:CV_vhx_vx_Lp_one_patch}\\
\zeta_{h,lr}\to \zeta_{lr}\qquad 	&\mbox{uniformly in }[0,T]\label{eq:CV_zetah_zeta_one_patch}
\end{align}
when $h=(\Delta x,\Delta t)\to 0$.
\label{theo:CV_solution_interfaces_single_patch}
\end{theo}
The rest of this section is devoted to the proof of Theorem~\ref{theo:CV_solution_interfaces_single_patch}, which closely follows \cite{DBH84}.\\

For \eqref{eq:CV_vh_v_one_patch} we show that there is at least one subsequence $v_{h'}$ converging to some limit $v^*$, and that for any such converging subsequence the limit $v^*$ is a solution to the Cauchy problem. By uniqueness $v^*=v$ and standard separation arguments this implies that the whole sequence $v_h\to v$ as in our statement.

By \eqref{eq:estimate_Linfty_Lipschitz_vh}-\eqref{eq:estimate_Holder_vh} with the upper bound \eqref{eq:estimate_interface_vh} for the supports, we can extract a subsequence $\{h'\}\subseteq\{h\}$ such that $v_{h'}\to v^*$ uniformly in $\overline{Q}_T$ for some limit $v^*\in\mathcal{C}(\overline{Q}_T)$. For any fixed $t>0$ we see by \eqref{eq:estimate_vxx_vt_measures} that $\partial_{x}v_h(\,.\,,t)$ is bounded in $BV(\R)$ (bounded variation) uniformly in $h'$. By standard compactness \cite{AFP00} in BV spaces there is a further subsequence $\partial_x v_{h''}(\,.\,,t)\to w^*$ in $L^1(\R)$. By continuity we get $w^*(\,.\,)=\partial_x v^*(\,.\,,t)$, so by uniqueness and separation we conclude that $\partial_xv_{h'}(\,.\,,t)\to \partial_xv^*(\,.\,,t)$ for all $t>0$. An easy application of Lebesgue's dominated convergence theorem with uniform bounds $|\partial_x v_h|\leq \gamma_0$ gives strong $L^{p}(Q_T)$ convergence for all $p\in[1,\infty)$ as in our statement.

We check now that the limit $v^*$ is indeed a solution to the Cauchy problem in the sense of Definition~\ref{defi:weak_sols_v}. Since $v_h^0(x)\to v^0(x)$ uniformly in $\R$ and $v^*$ is continuous up to $t=0$ the initial trace will be taken in the strong sense, and it is enough to check that
\begin{equation}
 \int\limits_{\R}v^*(x,\,.\,)\varphi(x,\,.\,)\Big|_{t_0}^{t_1}\mathrm{d}x
 + \int\limits_{t_0}^{t_1}\int\limits_{\R}\left\{-v^* \partial_t\varphi	+ \sigma(v^*)\partial_x v^*\partial_x\varphi+\Big(1-\sigma'(v^*)\Big)|\partial_x v^*|^2\varphi\right\}\mathrm{d}x\,\mathrm{d}t=0.
 \label{eq:weak_formulation_v_IBP}
 \end{equation}
for all $0<t_0\leq t_1\leq  T$ and test functions $\varphi\in\mathcal{C}^{\infty}_c(\overline{Q}_T)$. This weak formulation formally follows from $\partial_tv=\sigma(v)\partial_{xx}^2 v+|\partial_x v|^2$ after multiplying by $\varphi$ and integration by parts. Let now $\varphi_k^n:=\varphi(x_k,t^n)$, set $N_0:=\lfloor t_0/\Delta t\rfloor$ and $N_1:=\lfloor t_1/\Delta t\rfloor$, and consider the approximate Riemann sum
$$
S:=\sum\limits_{n=N_0}^{N_1-1}\left\{
\sum\limits_k\left[
\frac{v_k^{n+1}-v_k^n}{\Delta t}-\big(\sigma(v_k^n)+\eps\big)\frac{Av_k^n}{\Delta x^2}-\left|\frac{v_{k+1}^n-v_{k-1}^n}{2\Delta x}\right|^2
\right]\varphi_k^n\Delta x
\right\}\Delta t.
$$
By construction of our scheme the summand in $S$ is identically zero for $x_k\notin[\zeta_l^n,\zeta_r^n]$ and $x_k\in [x_{K_l(n)},x_{K_r(n)}]$. In the remaining boundary layers, which have thickness at most $s^n=|\zeta^n-x_{K(n)}|\leq 2\Delta x$ and where $v_{k}^n$ is linear, we have $|(v_k^{n+1}-v_k^n)/\Delta t|=\mathcal{O}(1)$ by Lemma~\ref{lem:dv/dt_interface} and $(\sigma(v_k^n)+\eps)\frac{Av_k^n}{\Delta x^2}=\mathcal{O}(\Delta x)\frac{w_{k+1}^n-w_k^n}{\Delta x}=\mathcal{O}(1)$. Here we used in particular that the artificial viscosity $\eps=\mathcal{O}(\Delta x)$. Thus we see that $S\to 0$ when $h'\to 0$. Summing by parts in $S$ one can get $S=S'\to 0$, where $S'$ is the discrete $\Delta x\Delta t$ Riemann sum corresponding to \eqref{eq:weak_formulation_v_IBP}. Using then the definition of the interpolation $v_{h'}$ in terms 
of $v_k^n$,
 the strong convergence $v_{h'}\to v^*$, the 
Lipschitz and H\"older regularity of $v_{h'}$ and the test function $\varphi$, it is easy to express $S'$ as the sum of $\mathrm{dx}\mathrm{dt}$ integrals over all triangles $L_k^n,U_k^n$, plus a remainder $o(1)$, and then send $h'\to 0$ in order to retrieve the weak formulation \eqref{eq:weak_formulation_v_IBP} for $v^*$ (note that $\sigma\in\mathcal{C}^1([0,\infty))$ and therefore $\sigma'(v_h)\to \sigma'(v)$ uniformly). We refer to \cite[pp. 480]{DBH84} for the details.\\

Turning now to the uniform convergence of the interfaces, we only argue for the right one and write $\zeta^n=\zeta^n_r,\zeta_h=\zeta_{h,r}$ and $K(n)=K_r(n)$ for simplicity (the proof for the left interface is exactly similar). From \eqref{eq:estimate_interface_vh} we see that $\zeta_{h'}$ is bounded in $W^{1,\infty}(0,T)$, so up to extraction of a further sequence if needed we may assume that $\zeta_{h'}\to\zeta^*$ uniformly in $[0,T]$ for some $\zeta^*$. This limit $\zeta^*$ is moreover monotone nondecreasing in $t$ with $\zeta^*(0)=\zeta(0)$, as the uniform limit of the nondecreasing functions $\zeta_h$ with $\zeta_h(0)=\zeta(0)$. We shall prove that the limit agrees with the true interface $\zeta^*=\zeta$, and the same separation argument as before will then show that the whole sequence actually converges.

Following again \cite{DBH84} we first need a technical result ensuring that, at a point $(\zeta^*(t_0),t_0)$ where the limit $\zeta^*$ is moving with positive speed, then $v^*(\,.\,,t_0)$ grows at least linearly in an interior neighborhood $[\zeta^*(t_0)-\delta,\zeta^*(t_0)]$:
\begin{lem}
Let $v^*,\zeta^*=\lim v_{h'},\zeta_{h'}$ as above and $\underline{z}(t)$ as in Lemma~\ref{lem:def_ODE_z}. Then
\begin{enumerate}
 \item[(i)]
 For any $0<t_0<t_0+\eta\leq T$ and $\delta>0$ there holds
 \begin{equation}
 \int_{t_0}^{t_0+\eta}v^*(\zeta^*(s)-\delta,s)\,\mathrm{d}s
 \geq
 \delta\big(\zeta^*(t_0+\eta)-\zeta^*(t_0)\big)-\delta^2\eta\underline{z}(t_0)
 \label{eq:integrated_linear_growth_interface}
 \end{equation}
 \item[(ii)]
 If $0<t_0<T$ is such that $d\zeta^*/dt(t_0)$ exists and is positive, then there is $\delta_0>0$ and $c>0$ such that
 \begin{equation}
 v^*(\zeta^*(t_0)-\delta,t_0)\geq c\delta 
 \label{eq:linear_growth_interface}
 \end{equation}
for all $\delta\in[0,\delta_0] $.
\end{enumerate}

\label{lem:linear_growth_interface} 
\end{lem}
This is somehow the converse statement of a well known fact for the so-called \emph{waiting-time phenomenon}: if $(\zeta(t_0),t_0)$ is a free-boundary point and the pressure grows at least linearly in $x$ in an interior neighborhood$\{v>0\}\cap B_r(\zeta(t_0))\times \{t_0\}$ then the free-boundary starts to move immediately (see e.g. \cite[Theorem 15.19]{Va07} for a stronger statement and simple proof in dimension $d=1$ for the pure PME nonlinearity). This explanation is of course an educated guess, as we do not know at this stage that $\zeta^*=\lim\zeta_{h'}$ is really the interface. Also note in (ii) that  $\zeta^*\in W^{1,\infty}(0,T)$ is differentiable a.e., and that the statement fails if $d\zeta^*/dt(t_0)=0$.
\begin{proof}
We first give a formal proof, keeping in mind that at the discrete level we enforced $d\zeta/dt=-\partial_xv$ at the interface and that the AB estimate $\partial_{xx}v(x,t)\geq \underline{z}(t)$ holds. Taking $h'\to 0$ we thus expect $d\zeta^*/dt(t_0)=-\partial_xv^*(\zeta^*(t_0),t_0)$, so that $v^*$ should indeed grow at least linearly $\partial_xv(\zeta^*(t_0),t_0)<0$ whenever the interface is moving $d\zeta^*/dt(t_0)>0$. In fact (ii) rigorously follows from (i): for whenever $\zeta^*$ is differentiable at $t_0$ with $d\zeta^*/dt(t_0)>0$ then dividing \eqref{eq:integrated_linear_growth_interface} by $\eta\to 0$ and discarding the $\delta^2=o(\delta)$ term for small $\delta>0$ yields \eqref{eq:linear_growth_interface} with $c\approx d\zeta^*/dt(t_0)>0$. Let us therefore also give a formal proof of (i): all regularity issues left aside and assuming that $v^*(\zeta^*(t),t)=0$, $d\zeta^*/dt=-\partial_xv^*(\zeta^*(t),t)$ and $\partial_{xx}
^2v(x,t)\geq \underline{z}(t)$ as expected, we first integrate by parts and use the generalized Aronson-B\'enilan estimate to estimate
\begin{align*}
v^*(\zeta^*(s)-\delta,s)		& =
\underbrace{v^*(\zeta^*(s),s)}_{=0}-\int_{\zeta^*(s)-\delta}^{\zeta^*(s)}\partial_xv^*(x,s)\mathrm{d}x\\
  & =-\int_{\zeta^*(s)-\delta}^{\zeta^*(s)}\left(\partial_xv^*(\zeta^*(s),s)-\int_x^{\zeta^*(s)}\partial_{xx}^2v^*(y,s)\mathrm{d}y\right)\mathrm{d}x\\
  &	\geq\int_{\zeta^*(s)-\delta}^{\zeta^*(s)}\underbrace{-\partial_xv^*(\zeta^*(s),s)}_{=+d\zeta^*/dt(s)}\mathrm{d}x
  + \int_{\zeta^*(s)-\delta}^{\zeta^*(s)}\left(\int_x^{\zeta^*(s)}\underline{z}(s)\mathrm{d}y\right)\mathrm{d}x\\
    & \geq \delta\frac{d\zeta^*}{dt}(s)+\int_{\zeta^*(s)-\delta}^{\zeta^*(s)}\delta\underline{z}(s)\,\mathrm{d}x=\delta\frac{d\zeta^*}{dt}(s)+\delta^2\underline{z}(s).
\end{align*}
Recalling also that $\underline{z}(t)$ is monotone increasing and integrating from $t_0$ to $t_0+\eta$ we conclude that
\begin{align*}
 \int_{t_0}^{t_0+\eta}v^*(\zeta^*(s)-\delta,s)\,\mathrm{d}s	& \geq
 \int_{t_0}^{t_0+\eta}\left(\delta\frac{d\zeta^*}{dt}(s)+\delta^2\underline{z}(s)\right)\mathrm{ds}\\
  & \geq \int_{t_0}^{t_0+\eta}\left(\delta\frac{d\zeta^*}{dt}(s)+\delta^2\underline{z}(t_0)\right)\mathrm{ds}\\
  &= \delta\big(\zeta^*(t_0+\eta)-\zeta^*(t_0)\big)-\delta^2\eta\underline{z}(t_0)
\end{align*}
as desired.\\
Following \cite[Lemma 3.4]{DBH84} we now briefly sketch how to get (i) rigorously, from which (ii) will follow as already explained. For fixed $\delta,\eta,t_0>0$ let $p=\lfloor \delta/\Delta x\rfloor$, $q=\lfloor \eta/\Delta x\rfloor$, and $N=\lfloor t_0/\Delta t\rfloor$. Recalling that $\frac{\zeta^{n+1-\zeta^n}}{\Delta t}=-\frac{v_{K(n)+1}^n-v_{K(n)}^n}{\Delta x}$ and summing by parts instead of integrating by parts as above, an explicit computation gives the discrete equivalent of \eqref{eq:integrated_linear_growth_interface}
$$
\sum\limits_{n=N}^{N+q-1}v_{K(n)-p}\Delta t\geq p\Delta x\big(\zeta^{N+q}-\zeta^N)\big)-(p\Delta x)^2(q\Delta t)\underline{z}(t^N).
$$
Sending $h'\to 0$ with uniform convergence $v_{h'}\to v^*$, $\zeta_{h'}\to \zeta^*$ and $x_{K(n)}\to \zeta^*(t)$ for $n=\lfloor t/\Delta t\rfloor$ finally allows to retrieve \eqref{eq:integrated_linear_growth_interface} and the proof is complete. 
\end{proof}
Back to the proof of \eqref{eq:CV_zetah_zeta_one_patch}, we recall that we only need to establish $\lim \zeta_{h'}=\zeta^*=\zeta$. From \eqref{eq:estimate_interface_vh} we have $v_{h'}(x,t)=0$ for all $x\geq \zeta_{h'}(t)+\Delta x$. As a consequence $v^*(x,t)=\lim v_{h'}(x,t) =0$ for all $x\geq \zeta^*(t)$, which shows by definition of $\zeta(t)=\zeta_r(t)=\sup\{x:\,v(x,t)>0\}$ that $\zeta^*(t)\geq \zeta(t)$.
Assuming by contradiction that there is $t_1>0$ for which $\zeta^*(t_1)>\zeta(t_1)$, we claim that there is $t_0\in (0,t_1)$ such that $\zeta^*(t_0)>\zeta(t_0)$ and $d\zeta/dt(t_0)>0$.\par
For if not, then arguing backwards in time starting from $t_1$ it is easy to see that either $\zeta^*(t)=cst=\zeta^*(t_1)$ for all $t\in[0,t_1]$, or there is $t_2\in(0,t_1)$ such that $\zeta^*(t)=cst=\zeta^*(t_1)$ for all $t\in[t_2,t_1]$ with $\zeta^*(t_2)=\zeta(t_2)$.
The first case would contradict $\zeta^*(0)=\zeta(0)$ since $\zeta^*(t_1)>\zeta(t_1)\geq \zeta(0)$.
In the second case, $\zeta\leq \zeta^*$ and the monotonicity of $\zeta$ show that $\zeta^*(t)=\zeta(t)=cst=\zeta^*(t_1)$ for all $t\in[t_2,t_1]$, thus contradicting $\zeta^*(t_1)>\zeta(t_1)$.

For any such $t_0$ Lemma~\ref{lem:linear_growth_interface} gives then $v(\zeta^*(t_0)-\delta)\geq c\delta >0$ for small $\delta$'s, and in particular choosing $0<\delta< \zeta^*(t_0)-\zeta(t_0)$ small enough there is a point $x_0=\zeta^*(t_0)-\delta>\zeta(t_0)$ such that $v(x_0,t_0)\geq c\delta>0$. This finally contradicts $\zeta(t_0)=\sup\{x:\,v(x,t_0)>0\}$ and ends the proof of Theorem~\ref{theo:CV_solution_interfaces_single_patch}.


%
%
\section{The hole-filling problem}
\label{section:two_patches}
In this section we consider the so-called hole-filling problem. We choose two compactly supported ``patches'' $\hat{v}^0(x),\check{v}^0(x)$ such that: (i) both $\hat{v}^0,\check{v}^0$ are $\gamma_0$-Lipschitz, (ii) $0\leq \hat{v}^0(x),\check{v}^0(x)\leq M$, and (iii) $\operatorname{supp}\hat{v}^0$ is at positive distance from $\operatorname{supp}\check{v}^0$ with
$$
\hat{\zeta}_l(0)<\hat{\zeta}_r(0)<\check{\zeta}_l(0)<\check{\zeta}_r(0).
$$
Defining
$$
v^0:=\max\{\hat{v}^0,\check v^0\}
$$
this means that $\operatorname{supp}v^0=\operatorname{supp}\hat{v}^0\cup \operatorname{supp}\check{v}^0$ has an internal hole of width $d_0=\check{\zeta}_l(0)-\hat{\zeta}_r(0)>0$ between $\operatorname{supp}\hat{v}^0$ and $\operatorname{supp}\check{v}^0$. Let $v(x,t),\hat{v}(x,t),\check{v}(x,t)$ be the solution of the Cauchy problem with initial data respectively $v^0(x),\hat{v}^0(x),\check{v}^0(x)$. We are interested here in computing a numerical approximation to $v(x,t)$. By noncontraction of the supports we know that $\hat{\zeta}_r(t)$ is nondecreasing, $\check{\zeta}_l(t)$ is nonincreasing, and because the interfaces propagate with finite speed at most $\gamma_0$ (which also follows from Section~\ref{section:one_patch}) the first time when the supports touch
$$
T^*=\sup\Big\{t\geq 0:\quad \hat{\zeta}_r(t)<\check{\zeta}_l(t)\Big\}\leq \infty
$$
is positive (possibly infinite). By uniqueness this implies that
$$
v=\max\{\hat{v},\check{v}\} \qquad\mbox{in }[0,T^*),
$$
so for $t\in[0,T^*)$ the support of $v$ still has an internal hole of width $d(t)=\check{\zeta}_l(t)-\hat{\zeta}_r(t)>0$. A well-known property of GPME is that ``once an interface starts moving it never stops'', see e.g. \cite[Lemma 14.20]{Va07} in any dimension for the pure PME nonlinearity and \cite[Corollary 15.23]{Va07} for a simple proof in dimension one. Since the internal interfaces were at positive distance at time $0$ this implies that, if and when they meet in finite time $\hat{\zeta}_r(T^*)=x^*=\check{\zeta}_l(T^*)$, at least one of the internal interfaces has started moving (otherwise the two would not meet) and is therefore still moving with positive speed. As a consequence at least one of the patches $\hat{v},\check{v}$ becomes instantaneously positive at $x=x^*$ for $t>T^*$, the comparison principle then implies $v(x^*,t)\geq\max\{\hat{v}(x^*,t),\check{v}(x^*,t)\}>0$, and the hole eventually disappears at $t=T^*$. Once the hole has filled the internal interfaces disappear, $\operatorname{supp}
v(\,.\,,
t)$ becomes a connected interval $[\zeta_l(t),\zeta_r(t)]$ containing the whole $[\hat{\zeta}_l(T^*),\check{\zeta}_r(T^*)]$, and $v$ does not 
equal $\max\{\hat{v},\check{v}\}$ anymore.

In section~\ref{section:one_patch} we described how to compute the approximate solution and interfaces when the initial datum consists in a single patch, which is exactly our assumption for each of $\hat{v}^0,\hat{v}^0$ separately. Using the results in the previous section we can therefore construct an approximation to each of the corresponding solutions $\hat{v},\check{v}$ and track all the resulting interfaces. We explain below how this previous one-patch algorithm can be naturally extended to the above case of two initial patches, while tracking all the interfaces (internal and external), detecting the hole-filling with accuracy, and solving past this time.
\begin{rmk}
We discuss here the case of two patches only for the ease of exposition, but the argument is easily adapted to any arbitrary number of initial patches at positive distance one from each other.
\end{rmk}

%
%
Roughly speaking, the algorithm goes as follows: starting from $\hat{v}^0_k,\check{v}^0_k$,  construct two independent sets of approximate solutions and interfaces $(\hat{v}_k^n,\hat{\zeta}^n_{l,r})$ and $(\hat{v}_k^n,\hat{\zeta}^n_{l,r})$ applying the one-patch scheme from Section~\ref{section:one_patch} separately to each patch. As long as the internal interfaces do not meet keep solving, and define $v_k^n=\max\{\hat{v}_k^n,\check{v}_k^n\}$. If the interfaces meet at $t=t^N$ then stop tracking the internal interfaces $\hat{\zeta}_r, \check{\zeta}_l$, define the external interfaces $\zeta^N_l:=\hat{\zeta}_l^N,\zeta^N_r:=\check{\zeta}_r^N$, and resume the computation applying the one-patch scheme to $v^{n}_k$ starting from $v^N_k$ at time $t^N$. More precisely,
\begin{algo}[Nnumerical scheme for the hole-filling]
Initialize $\hat{v}^0_k:=\hat{v}^0(x_k)$, $\check{v}^0_k:=\check{v}^0(x_k)$, $v^0_k:=\max\{\hat{v}^0_k,\check{v}^0_k\}$, as well as $\hat{\zeta}_{l,r}^0:=\hat{\zeta}_{l,r}(0)$, $\check{\zeta}_{l,r}^0:=\check{\zeta}_{l,r}(0)$, and $\zeta^0_l:=\hat{\zeta}_l^0$, $\zeta^0_r:=\check{\zeta}^0_r$. For fixed $T>0$ and while $t^n\leq T$, do:
\begin{enumerate}
 \item
 \label{item:algo_no_hitting}
 Apply the one-patch algorithm from section~\ref{section:one_patch} separately to $\hat{v}^{n},\hat{\zeta}_{l,r}^{n}$ and $\check{v}^{n},\check{\zeta}_{l,r}^{n}$ in order to predict $\hat{v}^{(n+1)'},\hat{\zeta}_{l,r}^{(n+1)'}$ and $\check{v}^{(n+1)'},\check{\zeta}_{l,r}^{(n+1)'}$. If the predicted internal interfaces are at least $\Delta x$ away $\check{\zeta}_l^{(n+1)'}-\hat{\zeta}_r^{(n+1)'}>\Delta x$, update $(n+1)'\to(n+1)$, set $v^{n+1}_k:=\max\{\hat{v}^{n+1}_k,\check{v}^{n+1}_k\}$, and repeat step 1. Otherwise define the numerical filling time $\tilde{T}^*:=t^n$, the external interfaces $\zeta_l^n:=\hat{\zeta}_l^n$ and $\zeta^n_r:=\check{\zeta}_r^n$, and go to step 2.
 \item
 \label{item:algo_after_hitting}
 Apply the one-patch algorithm from section~\ref{section:one_patch} to $v^{n},\zeta^{n}_{l,r}$ in order to construct $v^{n+1},\zeta^{n+1}_{l,r}$, and repeat Step 2.
\end{enumerate}
\label{algo:two_patches}
\end{algo}

Note that because all the interfaces propagate with numerical speed at most $\gamma_0$ and the internal ones are at initial distance $d(0)>0$, Step~\ref{item:algo_no_hitting} will be applied at least for $t^n\leq d(0)/2\gamma_0$ hence $\tilde{T}^*\geq d(0)/2\gamma_0$ uniformly in the mesh parameters. In case the hitting does occur for some $\tilde{T}^*\leq T$ then the numerical internal interfaces are not defined for later times. 

%
\subsection{A priori estimates}
We show here that all the previous estimates discrete are preserved across and after the filling time, including the $L^{\infty}$, Lipschitz, and H\"older bounds as well as the generalized Aronson-B\'enilan estimate. In particular we will obtain that the pressure $v$ stays $\gamma_0$-Lipshitz across the filling time, which is well known to hold in dimension one only (formally because $w=\partial_x v$ satisfies a maximum principle as in the proof of Lemma~\ref{lem:Linfty_Lipschitz_estimate}). As in the previous section we impose the \eqref{eq:CFL} condition on the mesh parameters $\Delta x,\Delta t,\eps$.
\begin{prop}
Let $v_k^n$ be the (two-patches) discrete solution constructed with Algorithm~\ref{algo:two_patches}, and $\underline{z}(t)<0$ as in Lemma~\ref{lem:def_ODE_z}. Then 
$$
0\leq v_k^n\leq M,
\quad
\left|\frac{v_{k}^n-v_{k-1}^n}{\Delta x}\right|\leq \gamma_0,
\quad
\left|v_{k}^n-v_k^m\right|\leq C|t^n-t^m|^{1/2},
\quad
\frac{Av_k^n}{\Delta x^2}\geq \underline{z}(t^n)
$$
hold for all $k$ and $t^n,t^m\in[0,T]$.
\label{prop:a_priori_estimates_two_patches}
\end{prop}
\begin{proof}
If no hole filling occurs our statement immediately follows from the results in Section~\ref{section:one_patch}, as $v_k^n$ coincides with either $\hat{v}_k^n$ or $\check{v}_k^n$, depending on which side of the internal hole one is looking at. Thus we may assume that internal interfaces meet at $t=t^N$.

For times $t^n\leq t^N$ the patches $0\leq \hat{v}_k^n,\check{v}_k^n\leq M$ are $\gamma_0$-Lipschitz (Lemma~\ref{lem:Linfty_Lipschitz_estimate}) so clearly $v^n_k=\max\{\hat{v}_k^n,\check{v}_k^n\}$ satisfies the same bounds for all $t^n\leq t^N$, and in particular at $t=t^N$. By definition $v^n_k$ is then constructed for $t^n\geq t^N$ by applying the one-patch scheme to solve the discrete Cauchy problem starting from the initial data $v^N_k$ at time $t^N$. Since $v_k^{N}$ satisfies the desired bounds we conclude by Lemma~\ref{lem:Linfty_Lipschitz_estimate} that $v_k^n$ satisfies the same $L^{\infty}$ and $\gamma_0$-Lipschitz estimates for all $t^n\geq t^N$.

Regarding now the H\"older continuity in time, we check that the proof of Proposition~\ref{prop:Holder} still applies. In Step 1 ($V_k^n\leq 0$ in $Q$ by induction on $n\in [n_0,n_1]$) the initialization $n=n_0$ only requires $\gamma_0$-Lipschitz bounds, which is true here. For the induction step we distinguished three cases: (i) $x_k$ is outside of the support with $v_k^n=0$ , (ii) $x_k$ is within one of the boundary layers, and (iii) when $v^{n+1}_k$ is constructed applying the finite difference scheme \eqref{eq:scheme}. All three cases are easily checked here with two patches: (i) and (ii) are identical, and (iii) also works here since $v^{n+1}_k$ is in fact constructed applying the finite difference equation \eqref{eq:scheme} to either one of the two patches before the filling time and to the unique patch afterward. Step 2 is identical, since it relies only on structural considerations and the previous $L^{\infty}$ and Lipschitz bounds.

We finally turn to the AB estimate. By definition of the hitting time $t^N$ we have that $\check{\zeta}_l^n-\hat{\zeta}_r^n>\Delta x$ stay strictly $\Delta x$ away from each other for $t^n\leq t^N$, so that there is always at least one integer mesh point in the hole. Since $v_k^n\geq 0$ everywhere and $v_k^n=0$ in the hole it is easy to check that $Av_k^n\geq 0$ for all $x_k$ such that $\hat{\zeta}^n_r-\Delta x\leq x_k\leq \check{\zeta}_l^n+\Delta x$, hence the AB estimate is trivially satisfied there (recall that $\underline{z}(t)<0$). Now outside the hole $Av_k^n$ equals either $A\hat{v}_k^n$ or $A\check{v}_k^n$, hence the AB estimate holds for all $t^n\leq t^N$ and including at $t=t^N$. Now for $t^n\geq t^N$ the solution $v^{n}_k$ is constructed applying the one-patch algorithm with initial datum $v^N_k$ at time $t^N$, which satisfies 
the AB estimate. By Lemma~\ref{lem:Aronson_Benilan_estimate} we conclude that the estimate also holds for all $t^n\geq t^N$ and the proof is complete.
\end{proof}
%
\subsection{Convergence of the approximate solutions and interfaces}
For fixed $T>0$ we denote $Q_T=\R\times(0,T)$ and $h=(\Delta x,\Delta t)$ as before. As in section~\ref{subsection:CV_one_patch} we define $v_h$ to be continuous and piecewise linear in all triangle $L_k^n,U_k^n$ according to \eqref{eq:def_interpolation_vh_meshpoints}. The external $\zeta_{h,lr}$ and internal $\hat{\zeta}_{h,r},\check{\zeta}_{h,l}$ interfaces are defined to be piecewise linear as in \eqref{eq:def_interpolation_interface_meshpoints}. Note that $\zeta_{h,lr}$ are defined up to $t=T$, while $\hat{\zeta}_{h,r}\leq \check{\zeta}_{h,l}$ are only defined up to the (numerical) filling time
\begin{equation}
T^*_h:=\max\{t^n:\quad \check{\zeta}^n_l-\hat{\zeta}_r^n>\Delta x\}
\label{eq:def_Th}
\end{equation}
(see Algorithm~\ref{algo:two_patches}). If no filling is numerically detected before the end of the computation we simply do not define $T^*_h$. In any case $\hat{\zeta}_{h,r},\check{\zeta}_{h,l}$ are respectively monotone nondecreasing and nonincreasing as long as they exist.
\begin{theo}
For fixed $T>0$ the numerical solution $v_h$ converges uniformly in $\overline{Q}_T$ to the unique solution $v$ when $h\to 0$.
\label{theo:CV_solution_double_patch}
\end{theo}
\begin{proof}
Note that the proof of \eqref{eq:CV_vh_v_one_patch} for the case of one patch only in Theorem~\ref{theo:CV_solution_interfaces_single_patch} only relies on: (i) the discrete estimates on $v_h$ uniformly in $h$ allowing to get compactness both for $v_h$ and $\partial_x v_h$, (ii) uniqueness for the Cauchy problem, (iii) the consistence of the finite difference equation \eqref{eq:scheme} inside the support, and (iv) the fact that all quantities involved in \eqref{eq:scheme} are of order $\mathcal{O}(1)$ inside the numerical boundary layers, see section~\ref{subsection:CV_one_patch} for the details. By Proposition~\ref{prop:a_priori_estimates_two_patches} this remains true in the case of two patches, thus allowing to conclude as in the proof of Theorem~\ref{theo:CV_solution_interfaces_single_patch}.
\end{proof}
The uniform convergence of the interfaces is now more delicate, as we need to distinguish between cases depending on whether the hole fills or not before the computation time $T$ . Roughly speaking, as long as the interfaces make sense the convergence follows as in the case of one patch only. We prove in particular that, if and when the \emph{numerical} filling occurs at time $t=T^*_h$, then $T^*_h$ is indeed a good approximation to the theoretical filling time $T^*$:
\begin{theo}
Fix $T>0$ and let $T^*$ be the theoretical hole-filling time. Then
\begin{enumerate}
 \item[(a)]
 If $T^*<T$ then there is a small $\delta_0>0$ such that the numerical hitting occurs at times $T^*_h\leq T-\delta_0$ for all $h\leq h_0$, and there exists $\lim\limits_{h\to 0}T^*_h=T^*$. Moreover
 $$
 \|\zeta_{h,l}-\zeta_l\|_{L^{\infty}(0,T)}+\|\zeta_{h,r}-\zeta_r\|_{L^{\infty}(0,T)}\to 0
 $$
 and
 $$
 \|\hat\zeta_{h,r}-\hat{\zeta}_r\|_{L^{\infty}(0,T^*-\eta)}+\|\check{\zeta}_{h,l}-\check{\zeta}_l\|_{L^{\infty}(0,T^*-\eta)}\to 0
 $$
 for any small $\eta>0$ fixed.
 \item[(b)]
 If $T^*\geq T$ then for all $\eta>0$ there exists $h_0(\eta)$ such that for all $h\leq h_0$ either no numerical hitting occurs before $t=T$, or does so at $T^*_h\geq T-\eta$. In particular for small $\eta$ the internal interfaces $\hat{\zeta}_{h,r},\check{\zeta}_{h,l}$ are defined at least for $t\leq T-\eta$. Moreover
 $$
 \|\zeta_{h,l}-\zeta_l\|_{L^{\infty}(0,T)}+\|\zeta_{h,r}-\zeta_r\|_{L^{\infty}(0,T)}\to 0
 $$
 and
 $$
 \|\hat\zeta_{h,r}-\hat{\zeta}_r\|_{L^{\infty}(0,T-\eta)}+\|\check{\zeta}_{h,l}-\check{\zeta}_l\|_{L^{\infty}(0,T-\eta)}\to 0
 $$
 for any small $\eta>0$ fixed.
\end{enumerate} 
\label{theo:CV_hitting_time}
\end{theo}
Practically speaking this means that if a hole-filling is detected numerically at $t=T^*_h$ then indeed $T^*_h$ is a good approximation to the theoretical filling time $T^*$, while if no hole-filling is detected before the end of the computation then one has simply not waited long enough to see the hole-filling, i-e $T^*\geq T$. In any case the numerical interfaces converge to the theoretical ones, both internal (as long as they exist) and external (up to $t=T$).

Before going into the details, it is worth pointing out that at the filling time there holds
\begin{equation}
0\leq \check{\zeta}_{h,l}(T^*_h)-\hat{\zeta}_{h,r}(T^*_h)\leq \mathcal{O}(\Delta x).
\label{eq:zetal=zetar_hole_filling_time}
\end{equation}
Indeed by \eqref{eq:def_Th} we have $T^*_h=t^N$ for some $N$, which according to Algorithm~\ref{algo:two_patches} is characterized by the fact that virtually computing one more step separately for each patch would result in $\check{\zeta}_l^{N+1}-\hat{\zeta}_r^{N+1}\leq \Delta x$. Recalling that any interface propagates with discrete speed at most $\gamma_0$ (Lemma~\ref{lem:discrete_speed_interfaces}) we see that indeed $0\leq \check{\zeta}_l^{N}-\hat{\zeta}_r^{N}	\leq (\check{\zeta}_l^{N+1}-\hat{\zeta}_r^{N+1})+2\gamma_0\Delta t\leq \Delta x+2\gamma_0\Delta t\leq \mathcal{O}(\Delta x)$ since $\Delta t=\mathcal{O}(\Delta x^2)$.
\begin{proof}[Proof of (a)]
We first show that the hole-filling always eventually occurs before the end of the computation if $h$ is small enough, i-e $T_{h}^*\leq T-\delta_0$ for some small $\delta_0>0$ as in our statement. Assuming by contradiction that this does no hold, then by definition of $T^*_h$ there is a discrete subsequence (not relabeled) such that either no numerical hitting occurs before $t=T$, or occurs for times $T^*_h\nearrow T$. In any case and by definition of the internal interfaces we can find a sequence of points $(x_h,t_h)$ such that $t_h\nearrow T$ and $x_h\in [\hat{\zeta}_{h,r}(t_h),\check{\zeta}_{h,l}(t_h)]$ with $v_h(x_h,t_h)$=0. By monotonicity of the interfaces we see that $x_h$ stays in the fixed compact set $[\hat{\zeta}_{r}(0),\check{\zeta}_{l}(0)]$, so up to extracting a further subsequence we can assume that $x_h\to x_0\in[\hat{\zeta}_{r}(0),\check{\zeta}_{l}(0)]$. By Theorem~\ref{theo:CV_solution_double_patch} we get
$$
v(x_0,T)=\lim\limits_{h\searrow 0}v_h(x_h,T^*_h)=0	\quad\mbox{for some }	x_0\in[\hat{\zeta}_{r}(0),\check{\zeta}_{l}(0)].
$$
We argue now for the theoretical solution and interfaces in order to get a contradiction. Because $T^*<T$ and the internal interfaces start at positive distance from each other they must meet for some $x^*=\hat{\zeta}_{r}(T^*)=\check{\zeta}_{l}(T^*)\in[\hat{\zeta}_{r}(0),\check{\zeta}_{l}(0)]$. Then necessarily one of them has started moving before $t=T^*$ (otherwise they would not meet). Once an interface starts moving it never stops, so at least one of the interfaces is really moving at $t=T^*$ and thus $\hat{v}(x^*,t)>0$ or $\check{v}(x^*,t)>0$ for all $t>T^*$. By the comparison principle $v\geq \max\{\hat{v},\check{v}\}$ is positive everywhere in $[\hat{\zeta}_{r}(0),\check{\zeta}_{l}(0)]$ for all $t>T^*$, in particular for $t=T>T^*$. This finally contradicts $v(x_0,T)=0$.

We claim now that $\lim\limits_{h\searrow 0}T^*_h=T^*$. Since $0\leq T^*_h\leq T-\delta_0$ for small $h$, we can extract a subsequence such that $T^*_{h'}\to\tilde{T}^*$ for some $\tilde{T}^*<T$. We prove that necessarily $\tilde{T}^*=T^*$, which will show that the whole sequence converges. Virtually keeping applying the one-patch algorithm separately to each of the patches $\hat{v}_{h'},\check{v}_{h'}$ after $t=T^*_{h'}$, we can naturally extend $\hat{\zeta}_{h',r},\check{\zeta}_{h',l}$ to all $t\in[0,T]$. By construction of our scheme these extended interfaces, still denoted $\hat{\zeta}_{h',r},\check{\zeta}_{h',l}$ with a slight abuse of notations, coincide with the internal interfaces for $v_{h'}$ up to the numerical filling time $T^*_{h}$, after which we stop tracking the true internal interfaces but the extended ones virtually still exist up to $t=T$. Applying Theorem~\ref{theo:CV_solution_interfaces_single_patch} we see that the extended interfaces $\hat{\zeta}_{h',r},\check{\zeta}_{h',l} \to \hat{
\zeta}_{r},\check{\zeta}_{l}$ uniformly in $[0,T]$, where $\hat{\zeta},\check{\zeta}$ are the interfaces of each patch $\hat{v},\check{v}$ considered as two independent solutions. Since $T^*_{h'}\to \tilde{T}^*$ we get by \eqref{eq:zetal=zetar_hole_filling_time} and uniform convergence that
$$
\hat{\zeta}_r(\tilde{T}^*)-\check{\zeta}_l(\tilde{T}^*)=\lim\limits_{h'\to 0}\left(\hat{\zeta}_{h',r}(T^*_{h'})-\check{\zeta}_{h',l}(T^*_{h'})\right)=0.
$$
Since $\hat{\zeta}_{r},\check{\zeta}_{l}$ are monotone and start at positive distance, and because once an interface starts moving it never stops, they can only meet at a unique time. By definition this time is $t=T^*$, thus $\tilde{T}^*=T^*$ and $T^*_{h}\to T^*$ as desired.

Uniform convergence of the interfaces can be obtained as in the proof of Theorem~\ref{theo:CV_solution_interfaces_single_patch} as long as the internal interfaces exist and are tracked numerically (this is why we need to step $\eta$ away from $T^*$ as in our statement, thus ensuring that the internal interfaces are numerically defined at least for fixed time intervals $[0,T^*-\eta]$), and the proof is achieved.
\end{proof}
\begin{proof}[Proof of (b)] We claim that a hole-filling can only be detected numerically for times $T_h^*\geq T-\eta$ close to the total computation time $T$ if $h$ is small enough (and may actually not be detected). For if not, then $T^*_{h'}\leq T-\delta_0$ for some subsequence and fixed $\delta_0>0$. Arguing exactly as in (a) we conclude that $T^*_{h'}\to T^*$, which shows in particular that $T^*\leq T-\delta_0$ and contradicts $T^*\geq T$. The convergence of the interfaces is also exactly similar to the proof of Theorem~\ref{theo:CV_solution_interfaces_single_patch}, stepping again $\eta>0$ away from $t=T$ for the internal interfaces as in our statement.
\end{proof}

%
%
\section{Numerical experiments}
\label{section:num_exp_comments}
The stability \eqref{eq:CFL} condition was imposed in order to ensure Lipschitz bounds and $L^\infty$ stability of the scheme (Lemma~\ref{lem:Linfty_Lipschitz_estimate}), but also the generalized Aronson-B\'enilan estimate (Lemma~\ref{lem:Aronson_Benilan_estimate}). For numerical purposes the less stringent condition
\begin{equation}
\beta\leq \frac{1}{2\Big(\sigma(M)+\eps\Big)}
\quad\mbox{and}\quad
\gamma_0\Delta x\Big(1+S_1(M)/2\Big)\leq \eps\leq \mathcal{O}(\Delta x)
\label{eq:CFL'}
\tag{CFL'}
\end{equation}
suffices to guarantee the stability Lemma~\ref{lem:Linfty_Lipschitz_estimate} and seems to give satisfactory convergence (see below). Note that in contrast with \eqref{eq:CFL} this relaxed condition does not depend on $s_1(M),S_2(M)$ anymore. In any case the computationally expensive $\beta=\Delta t/\Delta x^2=\mathcal{O}(1)$ condition is necessary due to the explicit nature of the scheme. In \cite{H85} Hoff considered a linearly implicit version of \cite{DBH84} for the pure PME nonlinearity. We presented here the explicit scheme for the ease of exposition, but all the theoretical results in Sections~\ref{section:one_patch} and \ref{section:two_patches} extend to general nonlinearities by considering the implicit scheme
$$
\frac{v_k^{n+1}-v_k^n}{\Delta t}=\sigma(v_k^n)\frac{A v_k^{n+1}}{\Delta x^2}+\eps \frac{Av_k^n}{\Delta x^2}+\left|\frac{v_{k+1}^n-v_{k-1}^n}{2\Delta x}\right|^2.
$$
In this case the stability condition becomes $\Delta t=\mathcal{O}(\Delta x)$, which is clearly the best one can hope since the propagation law $d\zeta/dt=-\partial_x v$ is intrinsically hyperbolic.\\

In order to test our scheme and because no explicit solutions are known for general nonlinearities we restrict to the pure PME $\partial_tv =(m-1)v\partial_{xx}^2v+|\partial_x v|^2$, to which the
Barenblatt profiles
$$
t\geq -t_0:\qquad V_m(x,t;C,x_0,t_0)=\frac{1}{t_0+t}\left(C(t_0+t)^{2/(m+1)}-\frac{1}{2(m+1)}|x-x_0|^2\right)_+
$$
are explicit solutions for any $m>1$. Here $C>0$ is a free parameter, while $x_0,t_0$ reflect the invariance under shifts. The interfaces are then explicitly given by
$$
\zeta_{lr}(t)=x_0\pm \sqrt{2(m+1)C}\,(t_0+t)^{1/(m+1)}.
$$
For our numerical experiment we fix $m=2$ and choose arbitrary parameters
$$
\hat{v}(x,t):=V_2(x,t;4/6,0,1),\qquad \check{v}(x,t):=V_2(x,t;1/6,3\sqrt[3]{2},1)
$$
such that the initial supports of $\hat{v}^0(x):=\hat{v}(x,0),\check{v}^0(x):=\check{v}(x,0)$ are at positive distance from each other as in Section~\ref{section:two_patches}. The exact interfaces are
$$
\hat{\zeta}_{lr}(r)=0\pm 2(1+t)^{1/3},\qquad \check{\zeta}_{lr}(t)=3\sqrt[3]{2}\pm (1+t)^{1/3}.
$$
Starting with initial datum $v_0=\max\{\hat{v}^0,\check{v}^0\}$ the theoretical hole-filling time $T^*$ can be computed according to Section~\ref{section:two_patches} by solving $\hat{\zeta}_r(t)=\check{\zeta}_l(t)
\Leftrightarrow t=T^*$, which gives explicitly
$$
T^*=1,\qquad x^*=\hat{\zeta}_r(T^*)=\check{\zeta}_l(T^*)=2\sqrt[3]{2}\approx 2.5198.
$$

All the computations were performed on a personal computer with {\ttfamily Linux/Octave}. We only specify the value of $\Delta x$, the parameters $\Delta t,\eps$ being then chosen respectively with the largest and smallest value allowed by \eqref{eq:CFL'}.
Figure~\ref{fig:FIG_cauchy} shows a typical result with $\Delta x=0.01$ plotted for several values of $t$, and Figure~\ref{fig:FIG_interface} illustrates the corresponding numerical interfaces. The hole filling was numerically detected for $T^*_h=1.0205$ and $x^*_h=2.5236$ (compare with $T^*=1$ and $x^*=2.5198$).
%
\begin{figure}[h!]
 \begin{center}
 \def\svgwidth{1.\textwidth}
\begingroup%
  \makeatletter%
  \providecommand\color[2][]{%
    \errmessage{(Inkscape) Color is used for the text in Inkscape, but the package 'color.sty' is not loaded}%
    \renewcommand\color[2][]{}%
  }%
  \providecommand\transparent[1]{%
    \errmessage{(Inkscape) Transparency is used (non-zero) for the text in Inkscape, but the package 'transparent.sty' is not loaded}%
    \renewcommand\transparent[1]{}%
  }%
  \providecommand\rotatebox[2]{#2}%
  \ifx\svgwidth\undefined%
    \setlength{\unitlength}{615.2bp}%
    \ifx\svgscale\undefined%
      \relax%
    \else%
      \setlength{\unitlength}{\unitlength * \real{\svgscale}}%
    \fi%
  \else%
    \setlength{\unitlength}{\svgwidth}%
  \fi%
  \global\let\svgwidth\undefined%
  \global\let\svgscale\undefined%
  \makeatother%
  \begin{picture}(1,0.74902471)%
    \put(0,0){\includegraphics[width=\unitlength]{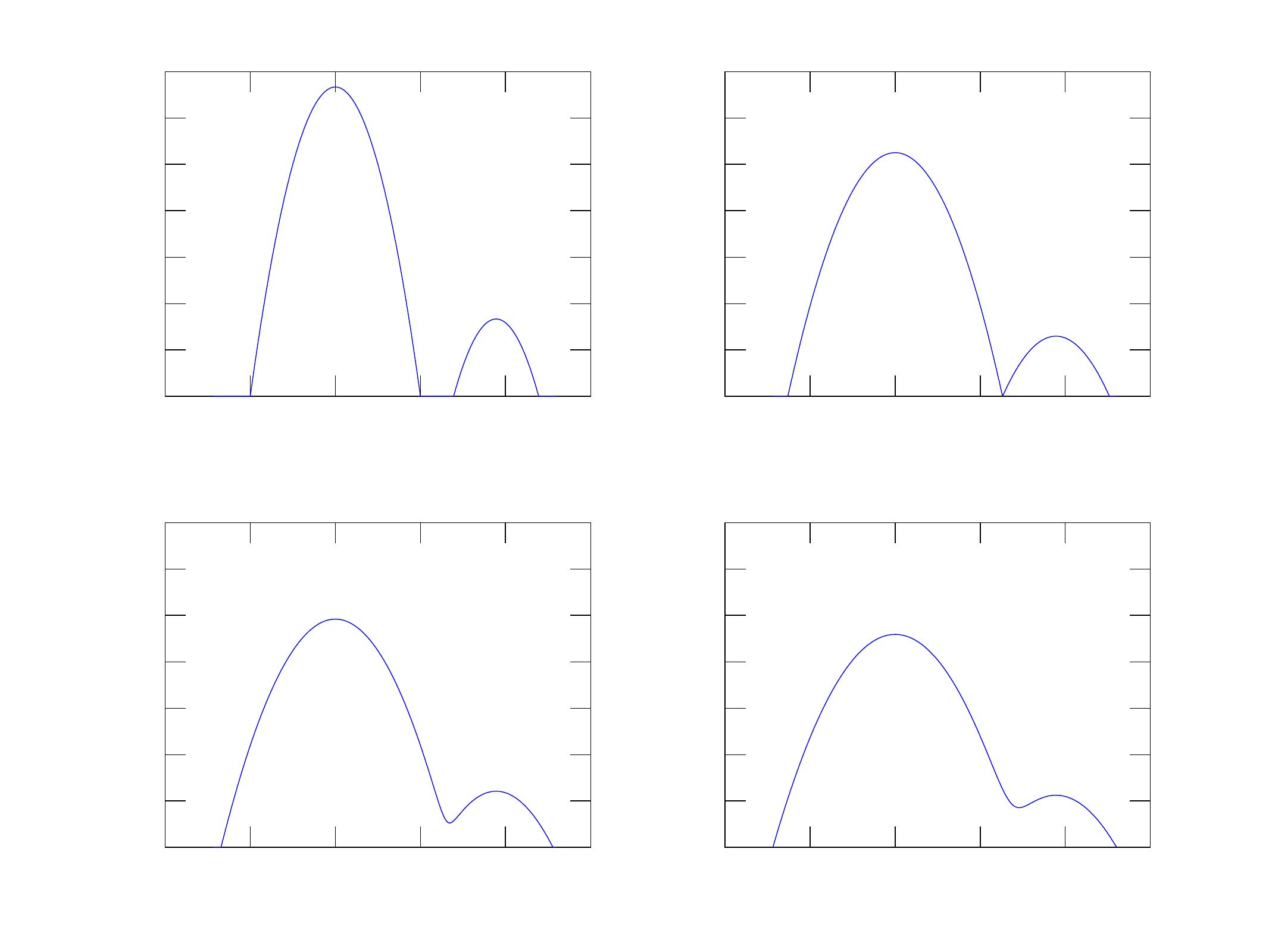}}%
    \put(0.55942783,0.07646294){\makebox(0,0)[rb]{\smash{0}}}%
    \put(0.55942783,0.1130039){\makebox(0,0)[rb]{\smash{0.1}}}%
    \put(0.55942783,0.14941482){\makebox(0,0)[rb]{\smash{0.2}}}%
    \put(0.55942783,0.18595579){\makebox(0,0)[rb]{\smash{0.3}}}%
    \put(0.55942783,0.22249675){\makebox(0,0)[rb]{\smash{0.4}}}%
    \put(0.55942783,0.25903771){\makebox(0,0)[rb]{\smash{0.5}}}%
    \put(0.55942783,0.29544863){\makebox(0,0)[rb]{\smash{0.6}}}%
    \put(0.55942783,0.3319896){\makebox(0,0)[rb]{\smash{0.7}}}%
    \put(0.57022107,0.05305592){\makebox(0,0)[b]{\smash{-4}}}%
    \put(0.63719116,0.05305592){\makebox(0,0)[b]{\smash{-2}}}%
    \put(0.70403121,0.05305592){\makebox(0,0)[b]{\smash{0}}}%
    \put(0.7710013,0.05305592){\makebox(0,0)[b]{\smash{2}}}%
    \put(0.83784135,0.05305592){\makebox(0,0)[b]{\smash{4}}}%
    \put(0.90481144,0.05305592){\makebox(0,0)[b]{\smash{6}}}%
    \put(0.73745124,0.01794538){\makebox(0,0)[b]{\smash{$x$}}}%
    \put(0.76065866,0.29208228){\makebox(0,0)[lb]{\smash{$t=2$}}}%
    \put(0.11911573,0.07646294){\makebox(0,0)[rb]{\smash{0}}}%
    \put(0.11911573,0.1130039){\makebox(0,0)[rb]{\smash{0.1}}}%
    \put(0.11911573,0.14941482){\makebox(0,0)[rb]{\smash{0.2}}}%
    \put(0.11911573,0.18595579){\makebox(0,0)[rb]{\smash{0.3}}}%
    \put(0.11911573,0.22249675){\makebox(0,0)[rb]{\smash{0.4}}}%
    \put(0.11911573,0.25903771){\makebox(0,0)[rb]{\smash{0.5}}}%
    \put(0.11911573,0.29544863){\makebox(0,0)[rb]{\smash{0.6}}}%
    \put(0.11911573,0.3319896){\makebox(0,0)[rb]{\smash{0.7}}}%
    \put(0.12990897,0.05305592){\makebox(0,0)[b]{\smash{-4}}}%
    \put(0.19687906,0.05305592){\makebox(0,0)[b]{\smash{-2}}}%
    \put(0.26371912,0.05305592){\makebox(0,0)[b]{\smash{0}}}%
    \put(0.33068921,0.05305592){\makebox(0,0)[b]{\smash{2}}}%
    \put(0.39752926,0.05305592){\makebox(0,0)[b]{\smash{4}}}%
    \put(0.46449935,0.05305592){\makebox(0,0)[b]{\smash{6}}}%
    \put(0.29713914,0.01794538){\makebox(0,0)[b]{\smash{$x$}}}%
    \put(0.32259081,0.29544863){\makebox(0,0)[lb]{\smash{$t=1.4472$}}}%
    \put(0.55942783,0.4313394){\makebox(0,0)[rb]{\smash{0}}}%
    \put(0.55942783,0.46788036){\makebox(0,0)[rb]{\smash{0.1}}}%
    \put(0.55942783,0.50429129){\makebox(0,0)[rb]{\smash{0.2}}}%
    \put(0.55942783,0.54083225){\makebox(0,0)[rb]{\smash{0.3}}}%
    \put(0.55942783,0.57737321){\makebox(0,0)[rb]{\smash{0.4}}}%
    \put(0.55942783,0.61391417){\makebox(0,0)[rb]{\smash{0.5}}}%
    \put(0.55942783,0.6503251){\makebox(0,0)[rb]{\smash{0.6}}}%
    \put(0.55942783,0.68686606){\makebox(0,0)[rb]{\smash{0.7}}}%
    \put(0.57022107,0.40793238){\makebox(0,0)[b]{\smash{-4}}}%
    \put(0.63719116,0.40793238){\makebox(0,0)[b]{\smash{-2}}}%
    \put(0.70403121,0.40793238){\makebox(0,0)[b]{\smash{0}}}%
    \put(0.7710013,0.40793238){\makebox(0,0)[b]{\smash{2}}}%
    \put(0.83784135,0.40793238){\makebox(0,0)[b]{\smash{4}}}%
    \put(0.90481144,0.40793238){\makebox(0,0)[b]{\smash{6}}}%
    \put(0.75953654,0.6503251){\makebox(0,0)[lb]{\smash{$t=1.0205$}}}%
    \put(0.11911573,0.4313394){\makebox(0,0)[rb]{\smash{0}}}%
    \put(0.11911573,0.46788036){\makebox(0,0)[rb]{\smash{0.1}}}%
    \put(0.11911573,0.50429129){\makebox(0,0)[rb]{\smash{0.2}}}%
    \put(0.11911573,0.54083225){\makebox(0,0)[rb]{\smash{0.3}}}%
    \put(0.11911573,0.57737321){\makebox(0,0)[rb]{\smash{0.4}}}%
    \put(0.11911573,0.61391417){\makebox(0,0)[rb]{\smash{0.5}}}%
    \put(0.11911573,0.6503251){\makebox(0,0)[rb]{\smash{0.6}}}%
    \put(0.11911573,0.68686606){\makebox(0,0)[rb]{\smash{0.7}}}%
    \put(0.12990897,0.40793238){\makebox(0,0)[b]{\smash{-4}}}%
    \put(0.19687906,0.40793238){\makebox(0,0)[b]{\smash{-2}}}%
    \put(0.26371912,0.40793238){\makebox(0,0)[b]{\smash{0}}}%
    \put(0.33068921,0.40793238){\makebox(0,0)[b]{\smash{2}}}%
    \put(0.39752926,0.40793238){\makebox(0,0)[b]{\smash{4}}}%
    \put(0.46449935,0.40793238){\makebox(0,0)[b]{\smash{6}}}%
    \put(0.29713914,0.37282185){\makebox(0,0)[b]{\smash{$x$}}}%
    \put(0.32259081,0.6503251){\makebox(0,0)[lb]{\smash{$t=0$}}}%
  \end{picture}%
\endgroup%

 \end{center}
  \caption[Figure~\ref{fig:FIG_cauchy}]{numerical solution $v_h(\,.\,,t)$ plotted for several times ($\Delta x=0.01$)}
 \label{fig:FIG_cauchy}
 \end{figure}
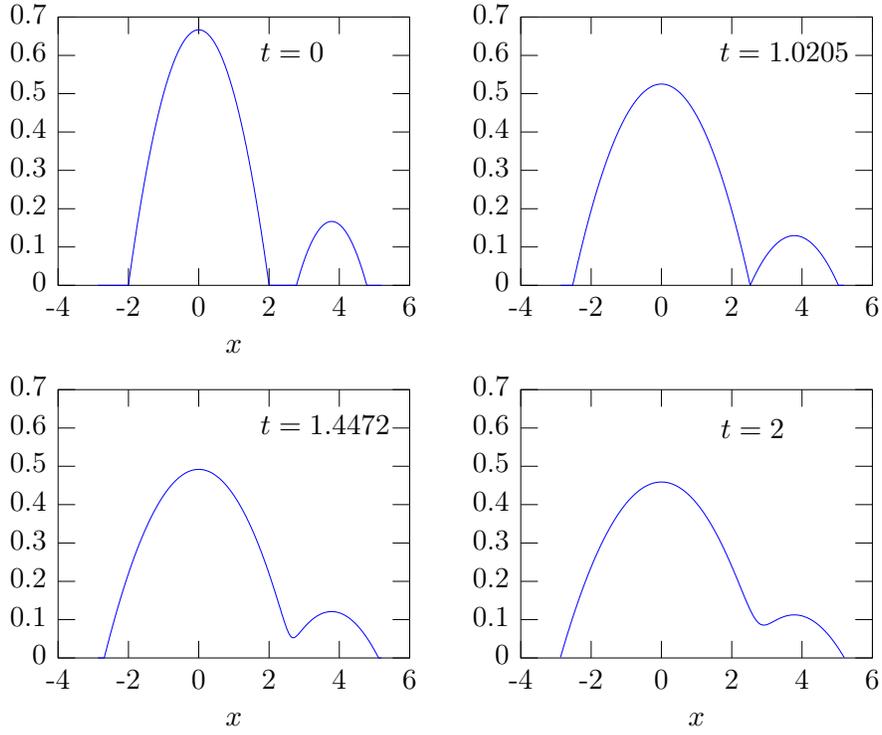
 \begin{figure}[h!]
 \centering
 \def\svgwidth{.7\textwidth}
\begingroup%
  \makeatletter%
  \providecommand\color[2][]{%
    \errmessage{(Inkscape) Color is used for the text in Inkscape, but the package 'color.sty' is not loaded}%
    \renewcommand\color[2][]{}%
  }%
  \providecommand\transparent[1]{%
    \errmessage{(Inkscape) Transparency is used (non-zero) for the text in Inkscape, but the package 'transparent.sty' is not loaded}%
    \renewcommand\transparent[1]{}%
  }%
  \providecommand\rotatebox[2]{#2}%
  \ifx\svgwidth\undefined%
    \setlength{\unitlength}{615.2bp}%
    \ifx\svgscale\undefined%
      \relax%
    \else%
      \setlength{\unitlength}{\unitlength * \real{\svgscale}}%
    \fi%
  \else%
    \setlength{\unitlength}{\svgwidth}%
  \fi%
  \global\let\svgwidth\undefined%
  \global\let\svgscale\undefined%
  \makeatother%
  \begin{picture}(1,0.74902471)%
    \put(0,0){\includegraphics[width=\unitlength]{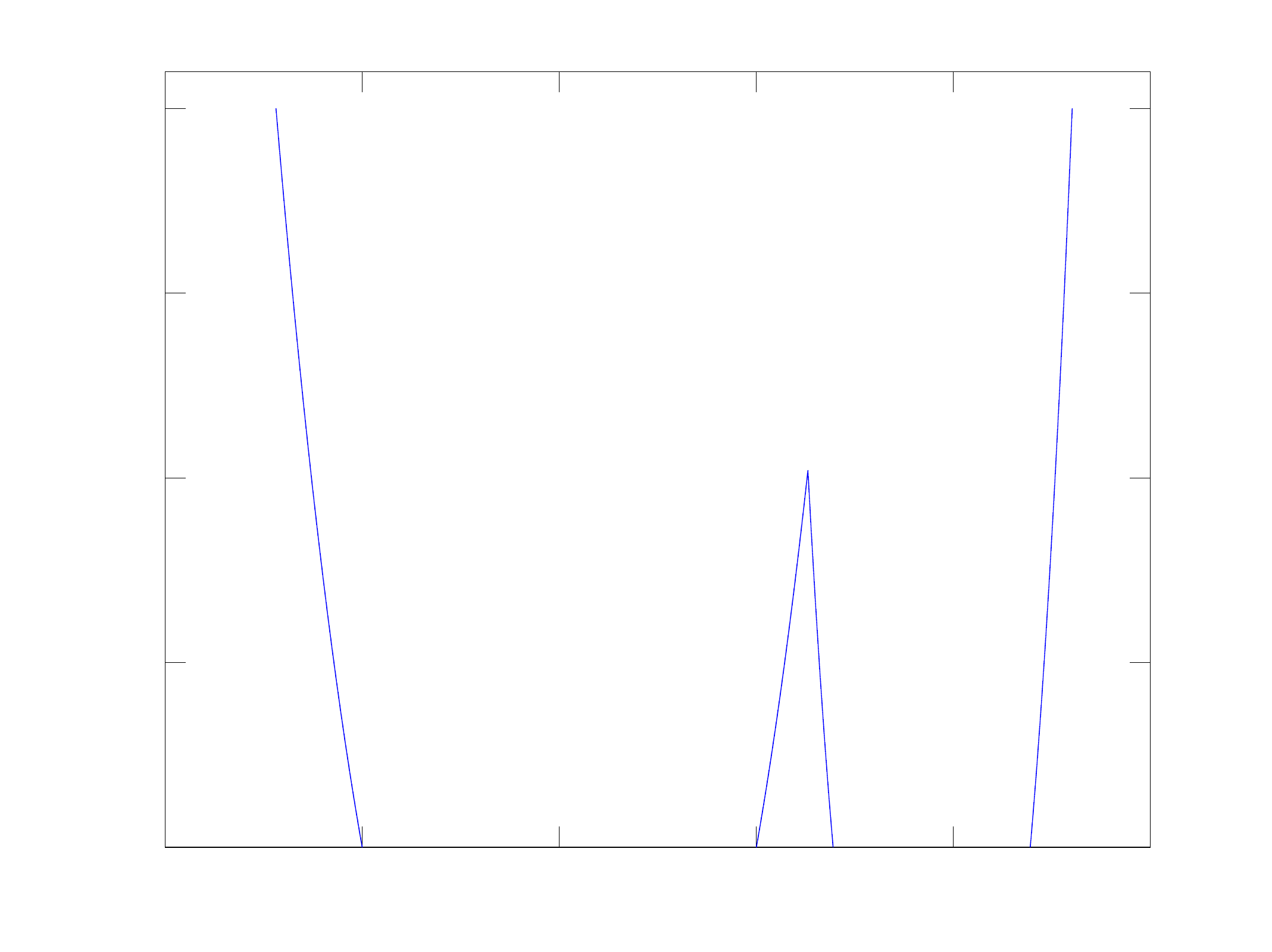}}%
    \put(0.11911573,0.07646294){\makebox(0,0)[rb]{\smash{0}}}%
    \put(0.11911573,0.22184655){\makebox(0,0)[rb]{\smash{0.5}}}%
    \put(0.11911573,0.36710013){\makebox(0,0)[rb]{\smash{1}}}%
    \put(0.11911573,0.51248375){\makebox(0,0)[rb]{\smash{1.5}}}%
    \put(0.11911573,0.65773732){\makebox(0,0)[rb]{\smash{2}}}%
    \put(0.12990897,0.05305592){\makebox(0,0)[b]{\smash{-4}}}%
    \put(0.28491547,0.05305592){\makebox(0,0)[b]{\smash{-2}}}%
    \put(0.43992198,0.05305592){\makebox(0,0)[b]{\smash{0}}}%
    \put(0.59479844,0.05305592){\makebox(0,0)[b]{\smash{2}}}%
    \put(0.74980494,0.05305592){\makebox(0,0)[b]{\smash{4}}}%
    \put(0.90481144,0.05305592){\makebox(0,0)[b]{\smash{6}}}%
    \put(0.07009103,0.38751625){\rotatebox{90}{\makebox(0,0)[b]{\smash{$t$}}}}%
    \put(0.51729519,0.01794538){\makebox(0,0)[b]{\smash{$x$}}}%
  \end{picture}%
\endgroup%

  \caption[Figure~\ref{fig:FIG_interface}]{interface curves ($\Delta x=0.01$)}
 \label{fig:FIG_interface}
 \end{figure}
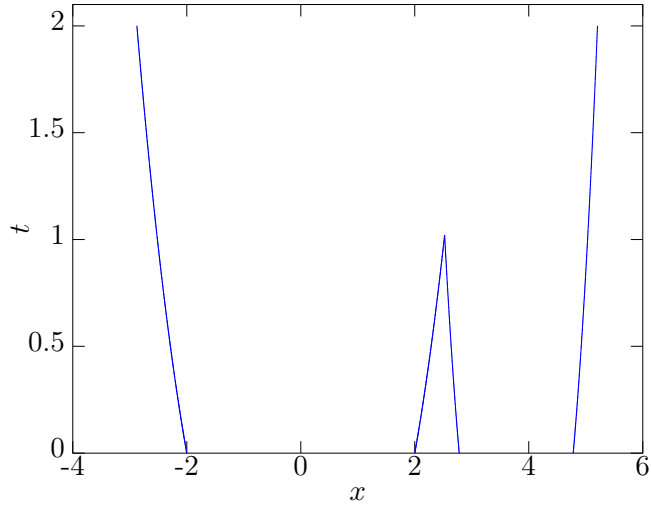
\par
In addition to an abstract convergence result as in Theorem~\ref{theo:CV_solution_interfaces_single_patch}, DiBenedetto and Hoff also derived explicit error estimates for the pure PME nonlinearity in the form $\|\zeta_h-\zeta\|_{L^\infty(0,T)}+\|v_h-v\|_{L^{\infty}(Q_T)}\leq \mathcal{O}\left( \Delta x^{\alpha}|\log \Delta x|^{\beta}\right)$ for some structural $\alpha,\beta$ related to $m>1$, see \cite[Theorem 4.1]{DBH84}. However their proof heavily relies on the explicit power structure $\Phi(s)=s^m$, and obtaining error estimates for general nonlinearities is a hard task that we did not carry out here due to the technical difficulties and lack of space. Figure~\ref{fig:FIG_error} shows the numerical errors $E_{x}:=\left|x^*_h-x^*\right|,E_t=|T^*_h-T^*|$ and $E_{\zeta}=\|\zeta_h-\zeta\|_{L^\infty(0,T^*_h))},E_v=\|v_h-v\|_{L^{\infty}(Q_{T^*_h})}$ 
as a function of $\Delta x$, and quite clearly exhibits $\mathcal{O}(\Delta x^{\alpha})$ convergence rates. Thus our scheme gives a good approximation of the solution, interfaces, and coordinates of the hole-filling as predicted from Theorem~\ref{theo:CV_solution_double_patch} and Theorem~\ref{theo:CV_hitting_time}.\\
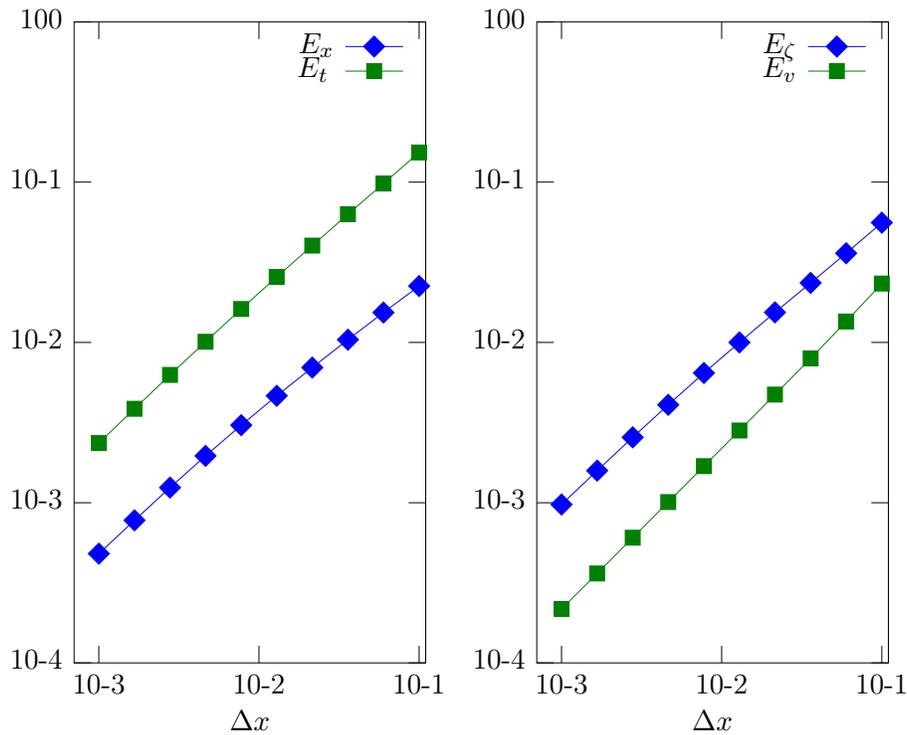
\begin{figure}[ht!]
 \def\svgwidth{1\textwidth}
 \centering

\begingroup%
  \makeatletter%
  \providecommand\color[2][]{%
    \errmessage{(Inkscape) Color is used for the text in Inkscape, but the package 'color.sty' is not loaded}%
    \renewcommand\color[2][]{}%
  }%
  \providecommand\transparent[1]{%
    \errmessage{(Inkscape) Transparency is used (non-zero) for the text in Inkscape, but the package 'transparent.sty' is not loaded}%
    \renewcommand\transparent[1]{}%
  }%
  \providecommand\rotatebox[2]{#2}%
  \ifx\svgwidth\undefined%
    \setlength{\unitlength}{615.2bp}%
    \ifx\svgscale\undefined%
      \relax%
    \else%
      \setlength{\unitlength}{\unitlength * \real{\svgscale}}%
    \fi%
  \else%
    \setlength{\unitlength}{\svgwidth}%
  \fi%
  \global\let\svgwidth\undefined%
  \global\let\svgscale\undefined%
  \makeatother%
  \begin{picture}(1,0.74902471)%
    \put(0,0){\includegraphics[width=\unitlength]{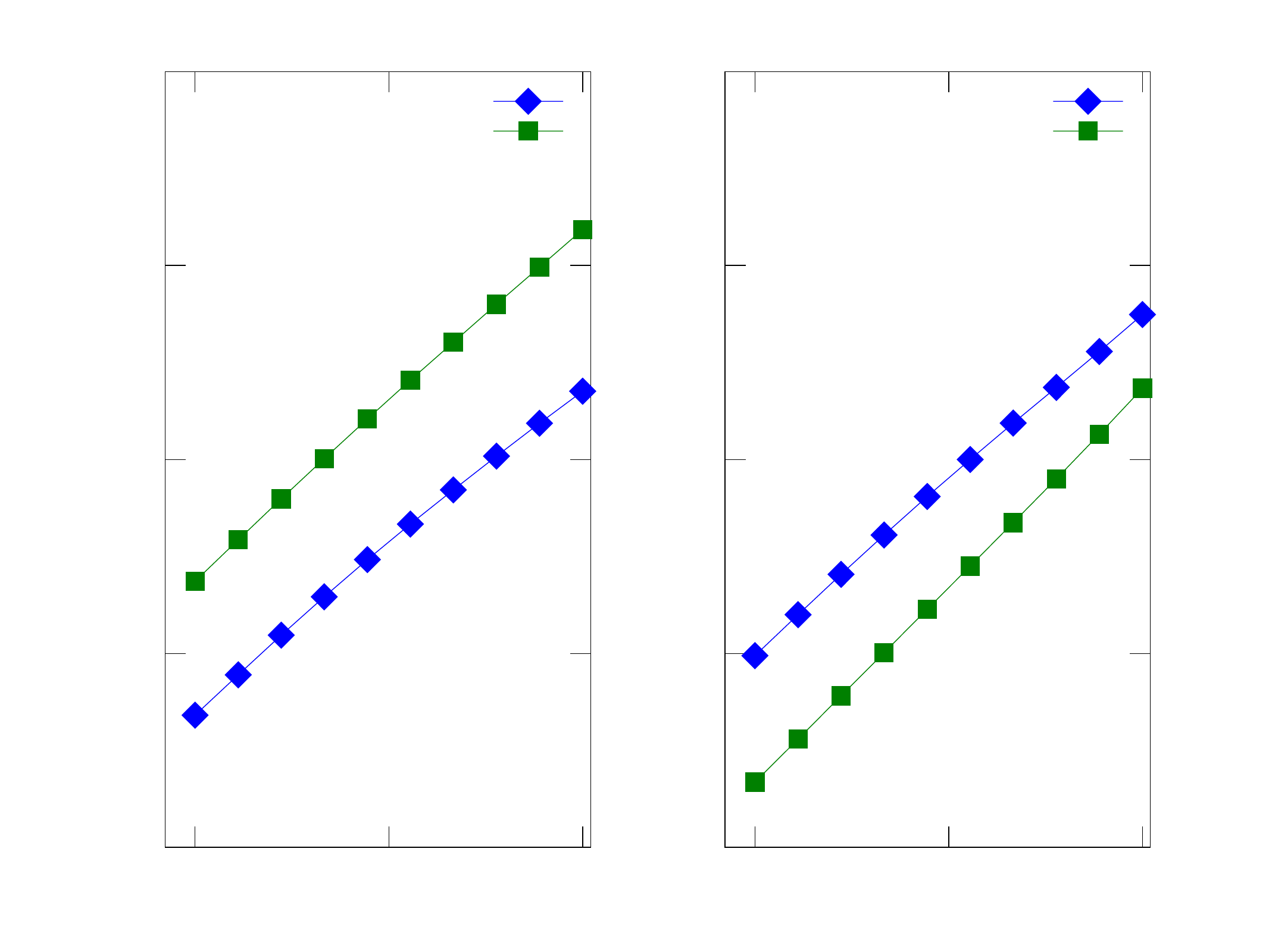}}%
    \put(0.55942783,0.07646294){\makebox(0,0)[rb]{\smash{10-4}}}%
    \put(0.55942783,0.2289987){\makebox(0,0)[rb]{\smash{10-3}}}%
    \put(0.55942783,0.3816645){\makebox(0,0)[rb]{\smash{10-2}}}%
    \put(0.55942783,0.5343303){\makebox(0,0)[rb]{\smash{10-1}}}%
    \put(0.55942783,0.68686606){\makebox(0,0)[rb]{\smash{100}}}%
    \put(0.59375813,0.05305592){\makebox(0,0)[b]{\smash{10-3}}}%
    \put(0.74616385,0.05305592){\makebox(0,0)[b]{\smash{10-2}}}%
    \put(0.89856957,0.05305592){\makebox(0,0)[b]{\smash{10-1}}}%
    \put(0.73745124,0.01794538){\makebox(0,0)[b]{\smash{$\Delta x$}}}%
    \put(0.81755527,0.66345904){\makebox(0,0)[rb]{\smash{$E_{\zeta}$}}}%
    \put(0.81755527,0.64005202){\makebox(0,0)[rb]{\smash{$E_v$}}}%
    \put(0.11911573,0.07646294){\makebox(0,0)[rb]{\smash{10-4}}}%
    \put(0.11911573,0.2289987){\makebox(0,0)[rb]{\smash{10-3}}}%
    \put(0.11911573,0.3816645){\makebox(0,0)[rb]{\smash{10-2}}}%
    \put(0.11911573,0.5343303){\makebox(0,0)[rb]{\smash{10-1}}}%
    \put(0.11911573,0.68686606){\makebox(0,0)[rb]{\smash{100}}}%
    \put(0.15344603,0.05305592){\makebox(0,0)[b]{\smash{10-3}}}%
    \put(0.30585176,0.05305592){\makebox(0,0)[b]{\smash{10-2}}}%
    \put(0.45825748,0.05305592){\makebox(0,0)[b]{\smash{10-1}}}%
    \put(0.29713914,0.01794538){\makebox(0,0)[b]{\smash{$\Delta x$}}}%
    \put(0.37724317,0.66345904){\makebox(0,0)[rb]{\smash{$E_x$}}}%
    \put(0.37168896,0.64005202){\makebox(0,0)[rb]{\smash{$E_t$}}}%
  \end{picture}%
\endgroup%

  \caption[Figure~\ref{fig:FIG_error}]{errors as a function of $\Delta x$}
 \label{fig:FIG_error}
 \end{figure}
\subsection*{Acknowledgements}
The author was supported by the Portuguese FCT fellowship SFRH/BPD/88207/2012.

\bibliographystyle{siam}
\bibliography{./biblio}

%
%
%
%
%
%
\end{document}